\documentclass[onefignum,onetabnum]{siamonline220329}

\usepackage{lipsum}
\usepackage{amsfonts}
\usepackage{graphicx}
\usepackage{epstopdf}
\usepackage{algorithmic}
\usepackage{booktabs}
\usepackage{comment}

\ifpdf
  \DeclareGraphicsExtensions{.eps,.pdf,.png,.jpg}
\else
  \DeclareGraphicsExtensions{.eps}
\fi

\usepackage{enumitem}
\setlist[enumerate]{leftmargin=.5in}
\setlist[itemize]{leftmargin=.5in}

\newsiamremark{remark}{Remark}
\newsiamremark{hypothesis}{Hypothesis}
\crefname{hypothesis}{Hypothesis}{Hypotheses}
\newsiamthm{claim}{Claim}
\newsiamremark{assumption}{Assumption}
\crefname{assumption}{Assumption}{Assumptions}

\newcommand{\R}{\mathbb{R}}
\newcommand{\N}{\mathbb{N}}

\newcommand{\E}{\mathbb{E}}
\allowdisplaybreaks

\headers{Early Stopping of Untrained Convolutional Neural Networks}{T. Jahn and B. Jin}

\title{Early Stopping of Untrained Convolutional Neural Networks\thanks{Submitted to the editors DATE.
\funding{The work of TJ was partly supported by the Deutsche Forschungsgemeinschaft under Germany’s Excellence Strategy - GZ 2047/1, Projekt-ID390685813 (Hausdorff Center for Mathematics) and EXC-2046/1, project ID 390685689 (the Berlin Mathematics Research Center MATH+), and that of BJ is supported by Hong Kong RGC General Research Fund (Project 14306423) and a start-up
fund from The Chinese University of Hong Kong.}}}

\author{Tim Jahn\thanks{Institut f\"ur Mathematik, TU Berlin, Germany 
  (\email{jahn@tu-berlin.de}, \email{timjahn3@gmail.com}).}
\and Bangti Jin\thanks{Department of Mathematics, The Chinese University of Hong Kong, Shatin, N.T., Hong Kong, P.R. China
  (\email{b.jin@cuhk.edu.hk}, \email{bangti.jin@gmail.com}).}
}

\usepackage{amsopn}

\ifpdf
\hypersetup{
  pdftitle={Early Stopping of Untrained Convolutional Neural Networks},
  pdfauthor={T. Jahn and B. Jin}
}
\fi

\begin{document}

\maketitle

\begin{abstract}
In recent years, new regularization methods based on (deep) neural networks have shown very promising empirical performance for the numerical solution of ill-posed problems, e.g., in medical imaging and imaging science. Due to the nonlinearity of neural networks, these methods often lack satisfactory theoretical justification. In this work, we rigorously discuss the convergence of a successful unsupervised approach that utilizes untrained convolutional neural networks to represent solutions to linear ill-posed problems. Untrained neural networks are particularly appealing for many applications because they do not require paired training data. The regularization property of the approach relies solely on the architecture of the neural network instead. Due to the vast over-parameterization of the employed neural network, suitable early stopping is essential for the success of the method. We establish that the classical discrepancy principle is an adequate method for early stopping of two-layer untrained convolutional neural networks learned by gradient descent, and furthermore, it yields an approximation with minimax optimal convergence rates. Numerical results are also presented to illustrate the theoretical findings.
\end{abstract}



\section{Introduction}\label{s1}
In this work we consider the numerical solution of linear inverse problems
\begin{equation}\label{eqn:lin}
y =Ax,
\end{equation}
with $A\in\R^{m\times n}$ being an ill-conditioned matrix, $x\in\R^n$ the sought solution and $y\in\R^m$ the exact data. In practice, the exact data $y$ is unknown, and we have access only to a noisy measurement
\begin{equation*}
    y^\epsilon=y + \epsilon z^\epsilon,
\end{equation*}
where $\epsilon>0$ denotes the noise level and {\color{black}$z^\epsilon\in\R^m$} the noise direction with $\|z^\epsilon\|=1$ (with $\|\cdot\|$ denoting the Euclidean norm of a vector and spectral norm of a matrix). Typical, the matrix $A$ is high dimensional, and it arises in various practical applications, e.g., image deblurring and denoising, compressed sensing and computed tomography. Due to the inherent ill-conditioning of the matrix $A$ and the presence of noise in the measurement $y^\epsilon$, direct solution methods like LU factorization 
will only yield poor approximations to the exact solution $x^\dag$, even for a small noise level $\epsilon$, since the inevitable data noise is amplified in an uncontrolled manner. Because of this observation, successful solution algorithms typically include regularization either implicitly or explicitly, which are continuous approximations of the (pseudo)-inverse of $A$ \cite{EnglHankeNeubauer:1996,ItoJin:2015}. An approximation is then chosen dependent of the measured data $y^\epsilon$ and the noise level $\epsilon$. Classical regularization schemes, e.g., Tikhonov regularization, Landweber iteration or spectral cut-off are part of the wide class of linear filter-based regularization methods.

In recent years many new regularization methods have been developed based on the use of (deep) neural networks. While these methods often yield remarkable empirical results in a plenitude of challenging applications, they can hardly be assessed from a theoretical view point and often there  are no rigorous convergence guarantees or error estimates at hand (see \cite{MukherjeeHauptmann:2023,ScarlettEldar:2023,HabringHoller:2023} for recent surveys on the theoretical analysis of neural solvers for inverse problems). This is often attributed to the nonlinearity of neural networks and the resulting nonconvexity of the loss function. One prominent example for a new regularization method based on neural networks is deep image prior \cite{Ulyanov:2018}. In this method, a convolutional neural network (CNN) $G$ is fitted to the non-linear least squares objective
 \begin{equation}\label{e00}
 \mathcal{L}(C):=\tfrac{1}{2}\|y^\epsilon -AG(C)\|^2,
 \end{equation}
where $C$ denotes the weights of the generator $G$ (often taken to be U-net \cite{Ronneberger:2015unet}). This method is very attractive from the view point of practical applications in that no paired training data is needed to learn the neural network parameters, whereas in many applications paired training data is expensive to collect, if not infeasible at all. The regularization effect is purely based on a specifically chosen architecture for the generator $G$ (a.k.a. regularization by architecture \cite{Dittmer:2020}). Deep image prior was first proposed for canonical tasks in imaging sciences, e.g., image denoising, super-resolution and inpainting. More recently, the method and its variants have found interest in medical imaging, e.g., magnetic resonance imaging \cite{YooJinUnser:2021}, computed tomography \cite{baguer2020diptv,barbano2021education} and positron emission tomography \cite{GongQiLi:2018}.

{\color{black}In a standard CNN architecture, given the input tensor $B_0$, the channels in the layers are given by
\begin{equation*}
    {B}_{i+1} = \mathrm{cn}(\mathrm{ReLU}(U_iB_iC_i)),\quad i =0,\ldots, d-1,
\end{equation*}
and the output of the $d$th layer is formed as ${x}=B_dC_{d+1}$. The matrices $C_i\in \mathbb{R}^{k\times k}$ and $C_{d+1}\in\mathbb{R}^{k\times k_o}$ contain the weights of the CNN. $k$ is the number of channels and determines the number of parameters of the CNN, and $\mathrm{cn}(\cdot)$ performs normalization on all channels individually. The rectified linear unit (ReLU), defined as $\mathrm{ReLU}(t)=\max(t,0)$, acts as the nonlinear activation function, and is applied to the matrix $U_iB_iC_i$ entrywise. The matrix $U_i\in\mathbb{R}^{n_{i+1}\times n_i}$ is a circulant matrix which implements a convolution filter with a fixed kernel $u_i$, e.g., triangular kernel, and corresponds to down- or up-sampling. The simplest model is a CNN with only one hidden layer and one output channel. Then the generator becomes 
\begin{equation*}
    G(C_0)=\mathrm{ReLU}(UB_0C_0)c_1,
\end{equation*}
with the circulant matrix  $U\in\mathbb{R}^{n\times n}$ for the filter ${u}$. Note that scaling the $i$th column of $C_0$ with a nonnegative factor is equivalent to rescaling the $i$th entry of the weight $c_1$. So one may fix $c_1=v:=(1,\ldots,1,-1,\ldots,-1)^\top/\sqrt{k}$ (with the first $\lfloor k/2\rfloor$ entries being 1 and the remaining entries being $-1$). The pattern of $v$ can be arbitrarily shuffled, and the choice is different from the common choice of using random Gaussian / uniform noise in deep image prior. Next, since the matrix $B_0\in\mathbb{R}^{n\times n}$ is Gaussian, with probability one, it has full rank and spans the whole $\mathbb{R}^n$, and then optimizing over $C_0\in\mathbb{R}^{n\times k}$ in $C=B_0C_0$ is equivalent to optimizing over $C\in\mathbb{R}^{n\times k}$.}
Thus, a two-layer CNN may be recast
\begin{equation}\label{e0}
G(C)={\rm ReLU}(UC)v,
\end{equation}

In this work we build upon a previous analysis from Heckel and Soltanolkotabi \cite{HeckelSoltanolkotabi:2020denoising} and analyze the dynamics of gradient descent on convolutional generators $G(C)$ in \eqref{e0}. Specifically, we learn the weights $C$ by applying the standard gradient descent with a constant step size $\eta>0$ to {\color{black}the} non-linear least squares objective \eqref{e00}, starting from a random centred Gaussian realisation  $C_0$ with variance $\omega^2$. Due to the ill-conditioning of $A$, this iteration must not be proceeded ad infinitum. It has to be terminated appropriately, a technique commonly known under the name early stopping in the inverse problems and machine learning communities
\cite{EnglHankeNeubauer:1996}. By now it is widely accepted that early stopping is one of the key issues that {\color{black}has} to be addressed for over-parameterized models \cite{Barbano:2023subspace,Wang2021stopping}. To this end, we use the discrepancy principle \cite{Morozov:1966}, which follows the idea that the iteration should be stopped whenever the residual norm is approximately of the same size as the data error norm $\epsilon$. Let $C_\tau^\epsilon$ denote the iterate at the $\tau$th iteration for noisy data $y^\epsilon$ and $L>1$ be a fudge parameter. Then the stopping index $\tau_{\rm dp}^\epsilon$ by the discrepancy principle is defined as
\begin{equation}\label{eqn:dp}
\tau_{\rm dp}^\epsilon=\tau_{\rm dp}^\epsilon(y^\epsilon):=\min\left\{\tau\ge 0~:~\|AG(C_\tau^\epsilon) - y^\epsilon\|\le L \epsilon\right\}.
\end{equation}
Formally this principle is applicable whenever an estimate of the noise level $\epsilon$ is known, which often can be estimated from the data.

In this work, we aim at providing a thorough theoretical justification of the discrepancy principle \eqref{eqn:dp}. Note that the development of reliable stopping rules has been widely accepted as one of the most outstanding practical challenges associated with deep image prior type techniques \cite{Wang2021stopping,Barbano:2023subspace}.  In Theorem \ref{t0} we provide a convergence result: when the neural network is properly (randomly) initialized and its width is sufficiently large, the stopping rule \eqref{eqn:dp} is well defined, and under the canonical H\"{o}lder type source condition on the minimum-norm solution $x^\dag$, the obtained approximation $G(C_{\tau_{\rm dp}^\epsilon}^\epsilon)$ can achieve the optimal accuracy with a high probability. To the best of our knowledge, this is one of the first theoretical justifications for the use of the technique in the context of ill-conditioned linear inverse problems.

The analysis of the stopping rule \eqref{eqn:dp} follows closely the by now established paradigm of over-parameterized neural networks: for sufficiently wide neural networks, the nonlinear model stays close to a linear one at the initialisation, so is the dynamics of gradient descent. This power of over-parameterization idea has been widely employed in the machine learning community (see, \cite{JacotHongler:2018,AroraDuHuLi:2019,Lee:2019,ChizatBach:2019,CaoGu:2019} for some early works). {\color{black}In essence, the approach relies on concentrations inequalities for random variables, cf. Proposition \ref{s2:p1} and its proof for the details.} The analysis in this work builds on several existing analyses in the context of denoising and compressed sensing \cite{HeckelSoltanolkotabi:2020CS,HeckelSoltanolkotabi:2020denoising}, but strengthens the argument {\color{black}to prove that the iteration and the residual for the nonlinear and the linearized models are uniformly close for a suitably given maximum number $T_\epsilon$ of iterations, instead of for all iterations as the iteration number tends to infinity, cf. Theorem \ref{t0} and Remark \ref{rmk:res-comparison} below for the precise statement and comparison}. This improvement allows us to derive error bounds that are independent of the smallest singular values of the matrices $A$ and $\Sigma(U)$ (cf. \eqref{eqn:Sigma} below for the definition). Together with techniques for analyzing the classical Landweber method, it enables us to establish the desired convergence rate. Meanwhile, the mathematical analysis of neural networks as a regularizer is actively studied from various angles (see e.g., \cite{LiSchwab:2020,Arndt:2022,BianchiLi:2023, Buskulic:2023convergence} for an incomplete list). The difference of our work from these existing works \cite{LiSchwab:2020,Arndt:2022,BianchiLi:2023} lies in our focus on early stopping and regularizing property, especially in the lens of iterative regularization. {\color{black}Note that the issue of early stopping has been investigated empirically in several recent works \cite{ChengSheldon:2019,Jo:2020,Wang2021stopping,Barbano:2023subspace,Nittscher:2024}. For example, Cheng et al \cite{ChengSheldon:2019} established an asymptotic equivalence of deep image prior to a stationary Gaussian process prior in the limit as the number of channels in each layer of the neural network goes to infinity, and based on the connection, the authors also proposed using stochastic gradient Langevin dynamics for posterior inference to alleviate the need of early stopping. In this work, we analyze the convergence behavior of the classical discrepancy principle \cite{Morozov:1966}. }

The rest of the paper is organized as follows. In Section \ref{s01}, we describe the main result, explain the main idea of the proof, and put the result into the context of iterative regularization. In Section \ref{s2}, we derive several preliminary estimates which are crucial to the convergence analysis, and in Section \ref{s3}, we give the complete proof of Theorem \ref{t0}. Finally, some numerical results are presented in Section \ref{s4} to complement the theoretical analysis.

 \section{Main result and discussions}\label{s01}

Now we present the main theoretical result of the work, which gives the convergence rate of the approximation given by the convolutional generator $G(C_
{\tau_{\rm dp}^\epsilon}^\epsilon)$, learned by the standard gradient descent in conjunction with the discrepancy principle \eqref{eqn:dp}.
In a customary way, we define the canonical H\"older source condition as
\begin{equation}\label{eqn:source}
\mathcal{X}_{\nu,\rho}:=\big\{x ~:~ x=(A^\top A)^\frac{\nu}{2}v,~\|v\|\le \rho\big\}.
\end{equation}
This condition is classic and needed in order to derive explicit convergence rates of the constructed approximation \cite[Section 4.2]{EnglHankeNeubauer:1996}. Due to the smoothing property of the matrix $A$, the condition represents a  certain regularity constraint of the reference solution. The method will provide an approximation to the minimum norm solution $x^\dag$, defined by 
\begin{equation*}
    x^\dag = \arg\min_{x\in\mathbb{R}^n:Ax = y} \|x\|.
\end{equation*}
 
The following theorem is the main result of the work. In the statement, for simplicity, we have assumed that $\eta=1$ and $\|A\|,\|U\|,\|\Sigma(U)\| \le 1$. Here $\Sigma(U):=\E\left[\mathcal{J}(C)\mathcal{J} (C)^\top\right]\in \R^{n\times n}$, {\color{black}where $\mathcal{J}(C)$ denotes the Jacobian of the convolutional generator $G(C)$ with respect to its parameters $C$}. The notation $\mathbb{E}$ denotes taking expectation with respect to the underlying distribution for $C$, i.e., i.i.d. centered Gaussian with variance {\color{black}$\omega^2$.} Note that the dependence on the width $k$ of the neural network generator $G$, the input $v$ and the iterates $C_\tau$ is suppressed from the notation for better readability, and that the dimensions $n$ and $m$ are fixed throughout. 

\begin{theorem}\label{t0}
Let $\Sigma(U)$ and $A^\top A\in\R^{n\times n}$ have a common eigenbasis $(w_i)_{i=1}^n$ with corresponding polynomially decaying eigenvalues 
$\sigma_i^2$ and $\alpha_j^2$: {\color{black}$\frac{\sigma_i^2}{i^{-p}} \in [b_\Sigma,B_\Sigma]$  and $\frac{\alpha_j^2}{j^{-q}} \in [b_A,B_A]$ for some constants $B_A\ge b_A>0$ and $B_\Sigma \ge b_\Sigma>0$.} Let the minimum norm solution $x^\dag$ to problem \eqref{eqn:lin} {\color{black}belong to $ \mathcal{X}_{\nu,\rho}$, let the corresponding noisy data $y^\epsilon$ satisfy $\|y^\epsilon\|\ge \epsilon$, and choose the maximum number $T_\epsilon$ of iterations such that}
\begin{equation*}
    T_\epsilon\ge \max\left(\frac{q(1+\nu)B_A^\frac{q+p}{q}}{2e(p+q)b_A b_\Sigma}\left(\frac{2\rho}{(L-1)\epsilon}\right)^{\frac{2(p+q)}{q(1+\nu)}},\frac{4\|y^\epsilon\|}{(L-1)\epsilon}\right).
\end{equation*}
Then, for $0<\delta_\epsilon <\frac{1}{4}$ and the entries of the initial weight matrix $C_0\in\R^{n\times k}$ chosen as independent centred Gaussian with variance $\omega^2 = \frac{(L-1)\epsilon}{8\sqrt{8n\log\left(\frac{2n}{\delta_\epsilon}\right)}}$ as well as
\begin{equation*}
k_\epsilon\ge \frac{2^{35} \|y^\epsilon\|^8 n\log\left(\frac{2n}{\delta_\epsilon}\right) T_\epsilon^{13}}{(L-1)^8\epsilon^8} ,
\end{equation*}
the following error estimate holds
\begin{equation*}
\sup_{x^\dag\in\mathcal{X}_{\nu,\rho}}\mathbb{P}\left(\|G(C^\varepsilon_{\tau^\varepsilon_{\rm dp}})-x^\dag\|\le \tilde{L}(\epsilon+ \epsilon^\frac{\nu}{\nu+1}\rho^\frac{1}{\nu+1} )\right)\ge 1-4\delta_\epsilon,
\end{equation*}
 for all $\epsilon \le \left(\frac{L-1}{16}\right)^\frac{1}{3}$, where the constant $\tilde{L}:= \left(2L\right)^\frac{\nu}{\nu+1} + \left(\frac{q(1+\nu) B_A^\frac{2q+p}{q}B_\Sigma}{2e(q+p)b_A^\frac{2q+p}{q}b_\Sigma}\right)^\frac{q}{2(q+p)}\left(\frac{4\rho}{L-1}\right)^\frac{1}{1+\nu} +L-1$ depends only on {\color{black}the parameters $B_A$, $b_A$, $B_\Sigma$, $b_\Sigma$, $p$, $q$, $\nu$ and $L$.}
 \end{theorem}

Theorem \ref{t0} indicates that the discrepancy principle \eqref{eqn:dp} is indeed a feasible stopping criterion for convolutional generators for solving linear inverse problems, and furthermore, under the standard H\"{o}lder source condition, the obtained estimate $G(C^\epsilon_{\tau^\epsilon_{\rm dp}})$ is minimax optimal. This rate is comparable with that for the standard Landweber method with the discrepancy principle for linear inverse problems \cite[Theorem 6.5]{EnglHankeNeubauer:1996}, apart from the probabilistic nature of the estimate, which is due to the use of suitable concentration inequalities. The result provides a first step towards establishing a complete theoretical foundation of deep image prior type methods, and resolves the challenge with reliable early stopping rules for such methods \cite{Wang2021stopping}. 
 
The analysis is based on the key observation that for a sufficiently wide neural network (i.e., for $k$ sufficiently large), the dynamics (iterate trajectory and the residuals) of the non-linear least squares problem \eqref{e00} stay close to a nearby linear least squares problem. This follows closely recent developments in machine learning \cite{JacotHongler:2018,AroraDuHuLi:2019,Lee:2019,ChizatBach:2019,CaoGu:2019}. For the proof of Theorem \ref{t0}, we proceed along the following four steps:
\begin{itemize}
\item First, we analyze in Theorem \ref{s2:t1} a general non-linear least squares problem \eqref{step1:e1}, and show that under Assumptions \ref{a1}-\ref{a3}, the iterates and residuals of \eqref{step1:e1} stay close to that of a nearby linearized problem \eqref{step1:e2}.
\item Second, in Proposition \ref{s2:p1}, we show how to pick the width $k$ and the variance $\omega^2$ of the initial weights $C_0$, such that the convolutional generator $G(C)$ \eqref{e00} fulfills Assumptions \ref{a1}-\ref{a3} with high probability.
\item Third, in Theorem \ref{s2:t2}, we combine Theorem \ref{s2:t1} and Proposition \ref{s2:p1} to obtain a general \textit{a priori} probabilistic error estimate for properly chosen $k$ and $\omega^2$.
\item Finally, we analyze the discrepancy principle \eqref{eqn:dp} as a data-driven stopping rule to finish the proof of Theorem \ref{t0}. This step will be carried out in Section \ref{s3}.
\end{itemize}

This analysis strategy differs substantially from that for analyzing iterative regularization techniques for nonlinear inverse problems, which typically relies on suitable nonlinearity condition, e.g., tangential cone condition (for convergence) and also range invariance condition (for convergence rates) \cite{KaltenbacherNeubauerScherzer:2008}. This difference stems from the fact that the approximation of neural networks can be made uniformly small over the domain when the neural network width is sufficiently large, relieving the range invariance condition completely. The analysis also reveals one delicate point in investigating the generator approach, of representing the unknown $x^\dag$ via a (nonlinear) generator: the spectral alignment between different singular vector bases, of the forward map $A$ and the generator $\Sigma(U)$ (population Jacobian). To resolve the challenge in the analysis, we have resorted to a spectral alignment assumption that $\Sigma(U)$ and $A^\top A$ share the eigenbasis (i.e., perfect alignment) and it would be of interest to relax the assumption. {\color{black}Note that due to the convolutional nature of $U$, $\Sigma(U)$ has the standard trigonometric basis \cite{HeckelSoltanolkotabi:2020denoising}, and thus the condition that $A^\top A$ and $\Sigma(U)$ share eigenvectors can be verified for integral equations with translation invariant kernels, e.g., image deblurring. The singular values of $\Sigma(U)$ depends on the discrete Fourier transform of the filter $u$, and the algebraic decay imposes certain restriction on the filter $u$. However, we believe that the assumption of algebraically decaying singular values is not essential for the proofs.} The numerical experiments in Section \ref{s4} indicate that the alignment assumption does affect much the performance of the discrepancy principle \eqref{eqn:dp}.

\begin{remark}
{\color{black} The required width $k_\epsilon$ depends on the noise level $\epsilon$ polynomially (i.e. $\epsilon^{-8}$), with a large constant $2^{35}$. It is expected that one may improve the constants by using  sharper concentration inequalities. Numerically, we observe the convergence behavior of the gradient descent type methods for fairly small neural networks, indicating that there is still room for significantly improving the involved constants.}   
 \end{remark}

 \begin{remark}
 {\color{black}Note that Theorem \ref{t0} gives a convergence rate. The rate depends only on the smoothing parameter $\nu$ in the source-wise representation \eqref{eqn:source}, and is independent of the spectral decay exponents $p$ and $q$. However, the number of the required iterations to reach the discrepancy principle \eqref{eqn:dp} depends on both spectral decay exponents $p$ (of the matrix $\Sigma(U)$) and $q$ (of the operator $A^\top A$). For any fixed $q$, the larger is the exponent $p$, the more iterations the method requires. Moreover, one still cannot conclude from Theorem \ref{t0} that the concerned method is a regularizing scheme in the lens of the classical regularization theory \cite{EnglHankeNeubauer:1996,KaltenbacherNeubauerScherzer:2008,ItoJin:2015}. Indeed, the probability in Theorem \ref{t0} depends on the noise level $\epsilon$, and thus the estimate is not uniform across all noise levels (down to zero) and the established argument \cite[Theorem 3.18]{EnglHankeNeubauer:1996} does not apply directly. One possibility to investigate regularizing properties is to utilize a general source condition \cite{mathe2008general}.}
 \end{remark}
 
\section{Preliminary estimates}\label{s2}
In this section, we develop the first three steps of the proof, one step in each subsection. {\color{black} First we develop a general theory of approximating a non-linear least-squares problem, by replacing a nonlinear map $f(\theta)$ by a linearized model $f_{\rm lin}=f(\theta_0)+J(\theta-\theta_0)$ for a suitable reference Jacobian $J$, in Section \ref{ssec:gen-nonlin}. The theory gives precise bounds on the iterates and residuals for the nonlinear and linearized models in terms of the approximation tolerance, cf. Theorem \ref{s2:t1} for the precise statement. This analysis is conducted under Assumptions \ref{a1}--\ref{a3}, which stipulates the approximation tolerance of the reference Jacobian $J$ and the local variation of the Jacobian $\mathcal{J}(\theta)$ in the parameter $\theta$. Then we verify the requisite conditions for the two-layer CNN generator $G(C)$, with the vectorized $C$ assuming the role of $\theta$ in Section \ref{ssec:conv-generator}. The main result is given in Proposition \ref{s2:p1}. This step invokes the concentration inequality to implicitly construct the reference Jacobian $J$ via the population covariance $\Sigma(U)$ of the Jacobian $\mathcal{J}(C_0)$ at the initial guess $C_0$. Last, using the fundamental estimates in Theorem \ref{s2:t1}, we derive a general error bound in Theorem \ref{s2:t2} for the convolutional generator $G(C)$ for solving linear inverse problems in Section \ref{ssec:general-bound}. This last step forms the basis for analyzing the a posterior stopping rule, i.e., the discrepancy principle \eqref{eqn:dp}.} 

\subsection{Theory for a general nonlinear least-squares problem}\label{ssec:gen-nonlin}
In the first step, we compare the following non-linear least squares problem
\begin{equation}\label{step1:e1}
 \mathcal{L}(\theta):=\tfrac{1}{2}\|A f(\theta) - y\|^2
 \end{equation}
  to the linearized least squares problem
\begin{equation}\label{step1:e2}
 \mathcal{L}_{\rm lin}(\theta):=\tfrac{1}{2}\|Af(\theta_0) + AJ(\theta-\theta_0) - y\|^2,
 \end{equation}
where $f:\R^N\to \R^n$ is a non-linear map with parameters $\theta\in\R^N$ and $J\in\R^{n\times N}$ is a fixed {\color{black} deterministic} matrix, called the reference Jacobian, which approximates the Jacobian {\color{black}$\mathcal{J}(\theta_0)=\nabla_\theta f(\theta)|_{\theta=\theta_0}$ (i.e., the Jacobian with respect to the parameter $\theta$) evaluated at} the starting point $\theta_0\in\R^N$, and for notational simplicity, below we denote the gradient $\mathcal{J}(\theta_\tau)$ by $\mathcal{J}_\tau$, i.e.,
$\mathcal{J}_\tau:=\mathcal{J}(\theta_\tau).$ {\color{black}Since $J$ in the linearized problem \eqref{step1:e2} is fixed throughout the iteration, it greatly facilitates the analysis of the linearized iteration below.}

The gradient descent updates starting from $\theta_0$ for problems \eqref{step1:e1} and \eqref{step1:e2} with a constant step size $\eta$ are respectively given by
\begin{align}\label{s2:e1}
\theta_{\tau+1} &= \theta_\tau- \eta \nabla \mathcal{L}(\theta_\tau) = \theta_\tau - \eta \mathcal{J}_\tau^\top A^\top (Af(\theta_\tau)-y) = \theta_\tau - \eta \mathcal{J}^\top _\tau A^\top r_\tau,\\
\label{s2:e2}
\tilde{\theta}_{\tau+1}&=\tilde{\theta_\tau} - \eta \nabla\mathcal{L}_{\rm lin}(\tilde{\theta}_\tau)=\tilde{\theta}_\tau - \eta J^\top A^\top(Af(\theta_0) + AJ(\theta_\tau-\theta_0) - y)=\tilde{\theta}_\tau -\eta J^\top A^\top \tilde{r}_\tau,
\end{align}
with the corresponding nonlinear and linear residuals given respectively by
\begin{align*}
r_\tau:=Af(\theta_\tau) - y\quad \mbox{and}\quad 
\tilde{r}_\tau:=Af(\theta_0) + AJ(\tilde{\theta}_\tau-\theta_0) - y.
\end{align*}
The linear residual $\tilde r_\tau$ admits a closed form expression. Indeed, by the definition of $\tilde\theta_\tau$,
\begin{align*}
\tilde r_\tau &=r_0+ AJ(\tilde{\theta}_\tau - \theta_0)\\
&= r_0 + AJ(\tilde{\theta}_{\tau-1} -\eta J^\top A^\top \tilde{r}_{\tau-1} - \theta_0)\\
&= r_0 + AJ(\tilde{\theta}_{\tau-1} -\theta_0) - \eta AJJ^\top A^\top \tilde{r}_{\tau-1}\\
& = \tilde r_{\tau -1} - \eta AJJ^\top A^\top \tilde{r}_{\tau-1} = (I-\eta AJJ^\top A^\top) \tilde{r}_{\tau-1}.
\end{align*}
Applying the recursion $\tau$ times yields the following expression for the linearized residual
\begin{align}\label{eqn:res-lin}
\tilde r_\tau =(I-\eta AJJ^\top A^\top)^\tau r_0.
 \end{align}
 Meanwhile, it follows from  the representation \eqref{s2:e2} that
 \begin{align}
     \tilde\theta_{\tau} &= \tilde \theta_{\tau-1} - \eta J^\top A^\top (I-\eta AJJ^\top A^\top)^{\tau-1} r_0\nonumber\\
    & = \tilde \theta_0-\eta J^\top A^\top \sum_{t=0}^{\tau-1} (I-\eta AJJ^\top A^\top)^{t}r_0\nonumber\\
    & =  \theta_0 - J^\top A^\top (AJJ^\top A^\top)^{-1}(I-(I-\eta AJJ^\top A^\top)^\tau)r_0,\label{eqn:theta-tilde}
 \end{align}
 where the $(AJJ^\top A^\top)^{-1}$ denotes the pseudo-inverse of the matrix, when it is not invertible.

Throughout, we denote by $\gamma\ge\|A\|$ an upper bound on the (spectral) norm of the matrix $A$. Further, we make the following assumptions on $\mathcal{J}(\theta)$ and the reference Jacobian $J$. {\color{black}These assumptions will be verified for a suitably over-parameterized convolutional generator $G(\theta)$ in Section \ref{ssec:conv-generator} below. Assumption \ref{a1} is about uniform boundedness of the Jacobian $\mathcal{J}(\theta)$, Assumption \ref{a2} the closeness of the initial Jacobian $\mathcal{J}_0$ to the reference one $J$, and Assumption \ref{a3} the uniform closeness of Jacobian $\mathcal{J}(\theta)$ to the initial one $\mathcal{J}(\theta_0)$. These assumptions are common for neural tangent kernel type analysis of neural networks \cite{JacotHongler:2018,ChizatBach:2019}. Note that one should not confuse the tolerance $\varepsilon$ with the noise level $\epsilon$.}
\begin{assumption}\label{a1}
There exists $\beta>0$ such that $\|J\|\le \beta$ and $\|\mathcal{J}(\theta)\|\le \beta$ for all $\theta\in\R^N$. 
\end{assumption}

\begin{assumption}\label{a2}
There exists $\varepsilon_0>0$ such that $\|\mathcal{J}_0 - J\| \le \varepsilon_0$ and $\|\mathcal{J}_0\mathcal{J}^\top_0 - JJ^\top\|\le \varepsilon_0^2$.
\end{assumption}

\begin{assumption}\label{a3}
There exist $\varepsilon>0$ and $R>0$ such that $\|\mathcal{J}(\theta)-\mathcal{J}(\theta_0)\| \le \frac{\varepsilon}{2}$ for all $\theta\in\R^N$ with $\|\theta-\theta_0\|\le R$.
 \end{assumption}

\begin{remark}
These assumptions differ slightly from the ones in Heckel \& Soltanolkotabi \cite{HeckelSoltanolkotabi:2020denoising}. More precisely, \cite{HeckelSoltanolkotabi:2020denoising} requires an additional restriction in Assumption \ref{a1} that the singular values $\sigma_i$ of the reference Jacobian $J$ are bounded by $\alpha$ and $\beta$ $($with $0<\alpha<\beta$$)$, i.e.,
$\alpha\le \sigma_n \le .. \le \sigma_1 \le \beta.$ Since the singular values $\sigma_i$ typically decay rapidly (which actually has to be employed in the existing results), the dependence of the derived estimates on {\rm(}inversely proportional to{\rm)} the smallest singular value $\sigma_{\min}$ is unsatisfactory. By an alternative way of proof, we remove this dependence. In the course of the proof, we shall also see that Assumption \ref{a1} can be relaxed to the condition $\|J\|\le \beta$ and $\|\mathcal{J}(\theta)\|\le \beta$ for all $\theta\in\mathbb{R}^N$ with $\|\theta-\theta_0\|\le R$.
\end{remark}

Now we show that the iterates $\theta_\tau$ and $ \tilde{\theta}_\tau$ and the corresponding residuals $r_\tau$ and $\tilde{r}_\tau$ stay close to each other throughout the iterations (up to the maximum number of iterations $T$).
\begin{theorem}\label{s2:t1}
Let Assumptions \ref{a1}, \ref{a2} and \ref{a3} be fulfilled, and furthermore, let the radius $R$ in Assumption \ref{a3} satisfy the following bound
\begin{equation*}
R\ge 2\|r_0\|\max\left( \eta\gamma\beta\left(1+2\eta \gamma^2(\varepsilon_0^2+\beta\varepsilon) T\right), \sqrt{T}\left(\sqrt{\eta} + \eta\gamma \sqrt{T}\left(\varepsilon+\varepsilon_0+\eta \gamma^2\beta(\varepsilon_0^2+\varepsilon\beta)T\right)\right)\right),
\end{equation*}
where the maximum number  $T\in\N$ of iterations fulfils $T\le \frac{1}{2\eta\gamma^2\varepsilon^2}$. Moreover, the constant step size $\eta$ obeys $\eta\le \frac{1}{\beta^2\gamma^2}$. Then for all $\tau\le T$, the following three estimates hold
\begin{align}\label{s2:t1:e1}
\|\theta_\tau-\theta_0\|&\le \tfrac{R}{2},\\\label{s2:t1:e2}
\|\tilde{r}_\tau-r_\tau\|&\le 2\eta \gamma^2(\varepsilon_0^2+\beta\varepsilon) \tau \|r_0\|,\\\label{s2:t1:e3}
\|\tilde{\theta}_\tau-\theta_\tau\|&\le \eta\gamma \left(\varepsilon+\varepsilon_0+\eta \gamma^2\beta(\varepsilon_0^2+\varepsilon\beta )\tau\right)\tau \|r_0\|.
\end{align}
\end{theorem}
\begin{proof}
We proceed along the lines of Heckel and Soltanolkotabi \cite{HeckelSoltanolkotabi:2020denoising}, using mathematical induction over $\tau\le T-1$. The statement for the base case $\tau=0$ holds trivially.

\medskip

\noindent{\bf Step 1: The next iterate $\theta_\tau$ fulfills $\|\theta_\tau - \theta_0\|\le R$.} By the triangle inequality, we have
\begin{equation*}
\|\theta_\tau- \theta_0\| \le \|\theta_\tau-\theta_{\tau-1}\| + \|\theta_{\tau-1}-\theta_0\| \le \|\theta_\tau - \theta_{\tau-1}\|  + \tfrac{R}{2}.
\end{equation*}
Since $\eta\le \frac{1}{\gamma^2\beta^2}$, we have $\|I-\eta A JJ^\top A^\top \|\le 1$, and thus $\|\tilde{r}_\tau\|\le \|r_0\|$. Consequently, by the definition of $\theta_\tau$, Assumption \ref{a1} and the induction hypothesis \eqref{s2:t1:e2} (for $\tau-1$),
\begin{align*}
\eta^{-1}\|\theta_\tau-\theta_{\tau-1}\|&=\|\nabla \mathcal{L}(\theta_{\tau-1})\|  = \|\mathcal{J}^\top_{\tau-1} A^\top  r_{\tau-1}\|\le \|\mathcal{J}^\top_{\tau-1} A^\top\|\left(\|\tilde{r}_{\tau-1}\| + \|\tilde{r}_{\tau-1} - r_{\tau-1}\|\right)\\
&\le \gamma\beta\left(\|\tilde{r}_{\tau-1}\| + \|r_{\tau-1}-\tilde{r}_{\tau-1}\|\right)\le \gamma\beta\|r_0\|\left(1+2\eta \gamma^2(\varepsilon_0^2+\beta\varepsilon) T\right),
\end{align*}
and we deduce the assertion of {\bf Step 1} from the definition of the radius $R$ that
$\|\theta_\tau-\theta_{\tau-1}\| \le \tfrac{R}{2}$.

\medskip

\noindent{\bf Step 2: Proof of \eqref{s2:t1:e2}.} From {\bf Step 1}, we have that $\|\theta_\tau-\theta_0\|\le R$ for all $\tau\le T$. Let $e_\tau:=r_\tau-\tilde{r}_\tau$. We follow Lemma 6.7 of Oymak et al \cite{OymakFabianLiSoltanolkotabi:2019} and set
$S_\tau:=\int_0^1 \mathcal{J}(\theta_\tau + t(\theta_{\tau+1}-\theta_\tau)){\rm d}t$.
Then by Assumption \ref{a1}, $\|S_\tau\|\leq \beta$. From the intermediate value theorem, we deduce
$f(\theta_{\tau+1}) = f(\theta_\tau) + S_\tau (\theta_{\tau+1}-\theta_\tau)$.
Thus, we obtain
\begin{align*}
r_{\tau+1}&=Af(\theta_{\tau+1}) -y = Af(\theta_\tau)-\eta AS_\tau\mathcal{J}^\top_\tau A^\top(Af(\theta_{\tau})-y) -y=\left(I-\eta A S_\tau \mathcal{J}^\top_\tau A^\top\right)r_\tau.
\end{align*}
Similarly, we have
$\tilde{r}_{\tau+1}=(I-\eta A J J^\top A^\top)\tilde{r}_\tau.$
Therefore, by the triangle inequality,
\begin{align*}
\|e_{\tau+1}\|&=\|(I-\eta A S_\tau \mathcal{J}^\top_\tau A^\top)r_\tau-(I-\eta A J J^\top A^\top)\tilde{r}_\tau\|\\
&\le \|I-\eta AS_\tau \mathcal{J}^\top_\tau A^\top\| \| r_{\tau}-\tilde{r}_\tau\| + \eta\left\|A S_\tau \mathcal{J}^\top_\tau A^\top - A J J^\top A^\top\right\|\|\tilde{r}_\tau\|.
\end{align*}
Next we bound the two terms separately.
First, by Assumption \ref{a3}, we have
\begin{equation*}
\|A\mathcal{J}_\tau\|,\, \|AS_\tau\|\le \gamma \beta \qquad\mbox{and}\qquad    \|AS_\tau - A\mathcal{J}_\tau\|\le \gamma \left(\|S_\tau-\mathcal{J}_0\|+\|\mathcal{J}_0-\mathcal{J}_\tau\|\right)\le \gamma\varepsilon.  
\end{equation*}
Thus from Lemma 6.3 of Oymak et al \cite[p. 21]{OymakFabianLiSoltanolkotabi:2019} we deduce
\begin{align*}
\|I-\eta A S_\tau \mathcal{J}^\top_\tau A^\top\|\leq1+\eta\gamma^2 \varepsilon^2.
\end{align*}
Second, by the triangle inequality and Assumptions \ref{a1}, \ref{a2} and \ref{a3}, we have
\begin{align*}
&\|AS_\tau \mathcal{J}^\top_\tau A^\top-AJJ^\top A^\top\|\\
&\qquad\le \gamma^2\left(\|S_\tau \mathcal{J}^\top _\tau- S_\tau\mathcal{J}^\top_0\| +\|S_\tau\mathcal{J}^\top_0 - \mathcal{J}_0 \mathcal{J}^\top_0\| + \|\mathcal{J}_0\mathcal{J}^\top_0-JJ^\top\|\right) \\
&\qquad\le \gamma^2\left(\|S_\tau\|\| \mathcal{J}_\tau-\mathcal{J}_0\| + \|\mathcal{J}_0\|\|S_\tau-\mathcal{J}_0\| + \|\mathcal{J}_0\mathcal{J}^\top_0-JJ^\top\|\right) \le \gamma^2(\beta\varepsilon+\varepsilon_0^2).
\end{align*}
Consequently, we deduce
\begin{equation}\label{s2:t1:e4}
\|e_{\tau+1}\| \le \eta \gamma^2 (\varepsilon_0^2+\beta\varepsilon)\|\tilde{r}_{\tau}\| + (1+\eta \gamma^2\varepsilon^2)\|e_\tau\|.
\end{equation}
Now we use the estimate $\|\tilde{r}_\tau\| \le \|r_0\|$. Since $e_0=r_0-\tilde{r}_0 = 0$, applying the inequality \eqref{s2:t1:e4} recursively yields
\begin{align*}
\|e_\tau\| &\le \eta\gamma^2(\varepsilon_0^2+\beta\varepsilon) \|r_0\| + (1+\eta\gamma^2 \varepsilon^2)\|e_{\tau-1}\|\\
&\le \eta\gamma^2(\varepsilon_0^2+\beta\varepsilon) \|r_0\| + (1+\eta \gamma^2\varepsilon^2)\left(\eta\gamma^2(\varepsilon_0^2+\beta\varepsilon) \|r_0\| + (1+\eta\gamma^2 \varepsilon^2)\|e_{\tau-2}\|\right) \\
&\le  ... \le \eta \gamma^2(\varepsilon_0^2 + \beta\varepsilon) \|r_0\|\sum_{j=0}^{\tau-1}(1+\eta\gamma^2\varepsilon^2)^j = \eta\gamma^2 (\varepsilon_0^2+\beta\varepsilon) \|r_0\| \frac{ (1+\eta\gamma^2\varepsilon^2)^\tau -1 }{\eta \gamma^2\varepsilon^2}.
\end{align*}
Using the inequalities $1+x\le e^x$ (for all $x\in\mathbb{R}$), $\log(1+x)\le x$ (valid for all $x>-1$), and the estimate $e^x\le \frac{1}{1-x}$ (valid for any $x<1$), we deduce
\begin{align*}
(1+\eta\gamma^2\varepsilon^2)^\tau &= e^{\tau\log\left(1+\eta\gamma^2\varepsilon^2\right)} \le e^{\tau \eta \gamma^2\varepsilon^2} \le \frac{1}{1-\tau \eta \gamma^2\varepsilon^2},
\end{align*}
since $\tau \eta \gamma^2\varepsilon^2 \le T\eta\gamma^2\varepsilon^2 \le \frac{1}{2}$, by the choice of $T$. Consequently,
\begin{align*}
\frac{(1+\eta \gamma^2\varepsilon^2)^\tau-1}{\eta \gamma^2\varepsilon^2} &\le \frac{\frac{1}{1-\tau \eta \gamma^2\varepsilon^2} - 1}{\eta \gamma^2\varepsilon^2} = \frac{\tau}{1-\tau \eta \gamma^2\varepsilon^2} \le \frac{\tau}{1-T\eta\gamma^2\varepsilon^2}\le 2\tau,
\end{align*}
which finally yields
$\|e_\tau\|\le 2\eta \gamma^2(\varepsilon_0^2+\beta\varepsilon) \|r_0\| \tau$,
as stated in \eqref{s2:t1:e2}.

\begin{remark}\label{rmk:res-comparison}
Note that using the lower bound $\alpha$ for the singular values of the reference Jacobian $J$ $($and $A=I$ and $\gamma=1$$)$, Heckel and Soltanolkotabi \cite[equation (34)]{HeckelSoltanolkotabi:2020denoising} derived the following uniform bound
\begin{equation}\label{eqn:bdd-HS}
    \|e_\tau\|\le 2 \alpha^{-2}\beta(\varepsilon_0+\varepsilon)\|r_0\|,\quad\forall \tau \in\mathbb{N}.
\end{equation}
In sharp contrast, the bound \eqref{s2:t1:e2} holds only for $\tau\leq T$. The difference stems from the ill-conditioning of $A$ and $J$.
The above bound \eqref{eqn:bdd-HS} might be much more pessimistic than the bound \eqref{s2:t1:e2}, depending on the concrete values of $\alpha$ and $\tau$.
\end{remark}

\medskip

\noindent{\bf Step 3: Proof of \eqref{s2:t1:e3}.} With $\tilde{\theta}_0=\theta_0$, applying the recursions \eqref{s2:e1} and \eqref{s2:e2} $\tau$ times yields
\begin{align*}
\eta^{-1}(\theta_\tau -\tilde{\theta}_\tau) &= \sum_{t=0}^{\tau-1}\left(\nabla \mathcal{L}(\theta_t) - \nabla \mathcal{L}_{\rm lin}(\theta_t)\right) = \sum_{t=0}^{\tau-1} \left(\mathcal{J}^\top_t A^\tau r_t - J^\top A^\top \tilde{r}_t\right).
\end{align*}
Now appealing to the estimate \eqref{s2:t1:e2} and then Assumptions \ref{a1} and \ref{a2} lead to
\begin{align*}
\eta^{-1}\|\theta_\tau -\tilde{\theta}_\tau\|&\le \sum_{t=0}^{\tau-1}\left(\|(\mathcal{J}^\top_t - J^\top)A^\top \tilde{r}_t\| + \|\mathcal{J}^\top_tA^\top(r_t-\tilde{r}_t)\|\right)\\
&\le \gamma\sum_{t=0}^{\tau-1}\left((\varepsilon_0+\varepsilon)\|\tilde{r}_t\| + \beta \|e_t\|\right)\\
&\le \gamma(\varepsilon_0+\varepsilon) \sum_{t=0}^{\tau-1}\|r_0\| + 2
\eta \gamma^3\beta(\varepsilon_0^2+\beta\varepsilon) \|r_0\| \sum_{t=0}^{\tau-1} t\\
&\le \gamma \tau \|r_0\|\left(\varepsilon+\varepsilon_0+\eta \gamma^2\beta(\varepsilon_0^2+\varepsilon\beta)\tau\right),
\end{align*}
which finishes the proof of {\bf Step 3}.

\begin{remark}
It is instructive to compare this bound to the following known bound $($for $A=I$ and $\gamma=1$ \cite[p. 22]{HeckelSoltanolkotabi:2020denoising}$)$:
$\|\theta_\tau-\tilde{\theta}_\tau\| \le \frac{(\varepsilon_0+\varepsilon)}{\eta \alpha^2}(1+2\eta \tau \beta^2)\|r_0\|$.
\end{remark}

\medskip

\noindent{\bf Step 4: Proof of \eqref{s2:t1:e1}.} Let $\sigma_i',w_i'$ be the eigenvalue decomposition of $A JJ^\top A^\top\in\mathbb{R}^{m\times m}$:
$$AJJ^\top A^\top w_i'= \sigma_i'^2  w_i',\quad i=1,\ldots,m,$$
where $w_i'\in\mathbb{R}^m$ are  orthonormal.
This, the representation \eqref{eqn:theta-tilde} and the estimate \eqref{s2:t1:e3} yield
\begin{align*}
\|\theta_\tau-\theta_0\| &\le \|\tilde{\theta}_\tau - \theta_0\| + \|\tilde{\theta}_\tau-\theta_\tau\|\\
&\le \sqrt{\sum_{i=1}^n(w_i',r_0)^2 \frac{(1-(1-\eta\sigma_i'^2)^\tau)^2}{\sigma_i'^2}} + \eta \gamma \tau \|r_0\|\left(\varepsilon+\varepsilon_0+\eta \gamma^2\beta(\varepsilon_0^2+\varepsilon\beta )\tau\right).
\end{align*}
Since  $\eta\sigma_i'^2\le \eta \gamma^2\beta^2 \le 1$, and in view of the inequality
\begin{equation}\label{inequality}
\sup_{0<\lambda\le 1} \lambda^{-\frac{1}{2}}|1-(1-\lambda)^\tau| \le \sqrt{\tau},
\end{equation}
cf. Lemma \ref{lem:basic-est} in the appendix, 
further from the choice of the radius $R$, we deduce
\begin{align*}
\|\theta_\tau -\theta_0\|&\le \sqrt{\eta\tau} \sqrt{\sum_{j=1}^n (w_i',r_0)^2} + \eta\gamma \tau \|r_0\|\left(\varepsilon+\varepsilon_0+\eta \gamma^2\beta(\varepsilon_0^2+\varepsilon\beta)\tau\right)\\
&\le \sqrt{\eta} \sqrt{T} \|r_0\| + \eta\gamma T \|r_0\|\left(\varepsilon+\varepsilon_0+\eta \gamma^2\beta(\varepsilon_0^2+\varepsilon\beta)T\right)\\
&=\sqrt{T}\left(\sqrt{\eta} + \eta\gamma \sqrt{T}\left(\varepsilon+\varepsilon_0+\eta \gamma^2\beta(\varepsilon_0^2+\varepsilon\beta)T\right)\right)\|r_0\| \le \frac{R}{2}.
\end{align*}
This concludes the proof of {\bf Step 4} and the induction step, and therefore also Theorem \ref{s2:t1}.
\end{proof}

In order to obtain simplified expressions for various bounds, from now on we will assume $\beta=\gamma=\eta=1$. For the general case, the results can be adjusted appropriately, using Theorem \ref{s2:t1}. First we give an easy corollary. These estimates will be used extensively below.
\begin{corollary}\label{s2:c1}
For $\beta=\gamma=\eta=1$, let Assumptions \ref{a1}, \ref{a2} and \ref{a3} hold with $\varepsilon_0=\varepsilon\le 1$, and fix
\begin{equation*}
R\ge 2\|r_0\|(\sqrt{T}+2\varepsilon T^2),
\end{equation*}
where the maximum number of iterations $T\in\N$ fulfils $T\le \frac{1}{2\varepsilon^2}$. Then, for all $\tau\le T$, the following bounds hold
\begin{align*}
\|\theta_\tau-\theta_0\|&\le \frac{R}{2},\quad
\|\tilde{r}_\tau-r_\tau\|\le 4\varepsilon \tau \|r_0\|,\quad \mbox{and}\quad \|\tilde{\theta}_\tau-\theta_\tau\|\le  2\varepsilon \tau^2 \|r_0\|.
\end{align*}
\end{corollary}
\begin{proof}
The corollary is direct from Theorem \ref{s2:t1}. Indeed, the bound on $\|r_\tau-\tilde r_\tau\|$ is already given in Theorem \ref{s2:t1}. The bound on $\|\tilde \theta_\tau-\theta_\tau\|$ follows from a small refinement of \textbf{Step 3} of the proof (with $\beta=\gamma=\eta=1$ and $\varepsilon=\varepsilon_0$) as 
\begin{equation*}
    \eta^{-1}\|\theta_\tau-\tilde\theta_\tau\| 
    \le \tau \|r_0\|\left(\varepsilon+\varepsilon_0+(\varepsilon_0^2+\varepsilon)(\tau-1)\right) \leq 2\epsilon\tau^2\|r_0\|.
\end{equation*}
With the given choice of $R$, the desired bound on $\|\theta_\tau-\theta_0\|$ follows analogously.
\end{proof}

\subsection{Verifying the assumptions for the convolutional generator}\label{ssec:conv-generator}
Now we turn back to the neural network generator $G(C)$ defined in \eqref{e0} with a fixed convolutional matrix $U\in\R^{n\times n}$ and $\|U\|,\|\Sigma(U)\|\le \beta =1$ as well as a constant step size $\eta=1$. Recall that
\begin{equation}\label{eqn:Sigma}
\Sigma(U)=\E\left[\mathcal{J}(C)\mathcal{J}(C)^\top\right]\in\mathbb{R}^{n\times n}, 
\end{equation}
{\color{black}where $\mathcal{J}(C)$ denotes the Jacobian of the convolutional generator $G(C)$, and the expectation $\mathbb{E}$ is taken with respect to the random neural network parameters $C$, whose entries are assumed to follow i.i.d. Gaussian with zero mean and variance $\omega^2$. The stochasticity of $C$ arises from the random initialization.}
Let $\delta>0$ and $0 < \xi \le \frac{1}{\sqrt{32\log\left(\frac{2n}{\delta}\right)}}$ be given, where $\delta$ is a tolerance parameter for a probability and $\xi$ is a tolerance parameter for the error. We first prove that for properly chosen $k\in\N$ and $\omega^2$, the neural network generator $G(C)$ in \eqref{e0} fulfills Assumptions \ref{a1}--\ref{a3}. That is, we can find a reference Jacobian $J$ with the designate approximation tolerance with high probability.
\begin{proposition}\label{s2:p1}
Let $\varepsilon, \varepsilon_0>0$. Then, for the choice
\begin{equation*}
k\ge \frac{16}{\varepsilon_0^4} \log\left(\frac{2n}{\delta}\right)
\quad \mbox{and}\quad
\omega = \frac{\xi \|y\|}{2\sqrt{8n\log\left(\frac{2n}{\delta}\right)}},
\end{equation*}
there exists a reference Jacobian $J\in\R^{nk\times k}$, such that for any $R\le \omega \tilde{R}$, with
$\tilde{R}:= \left(\frac{\varepsilon}{4}\right)^3\sqrt{k}\le \frac{1}{2}\sqrt{k}$,
Assumptions \ref{a1}, \ref{a2} and \ref{a3} are fulfilled with probability at least $1-\delta-n\exp(-\frac{\varepsilon^4}{2^9}k)$.
\end{proposition}
\begin{proof}
First, by the concentration lemma \cite[Lemma 3, p. 24]{HeckelSoltanolkotabi:2020denoising}, for a Gaussian matrix $C\in\R^{n\times k}$ with centered i.i.d. entries and variance $\omega^2$, there holds
\begin{equation*}
\|\mathcal{J}(C)\mathcal{J}(C)^\top  - \Sigma(U) \| \le \|U\|^2\sqrt{\log\left(\frac{2n}{\delta}\right) \sum_{\ell=1}^k v_\ell^4},
\end{equation*}
with probability at least $1-\delta$. {\color{black}That is, the random variable $\mathcal{J}(C)\mathcal{J}(C)^\top$ is close to its expectation $\Sigma(U)$ with high probability.} By the choice of  $k$ (and the condition $\|U\|\le1$), we deduce
\begin{equation*}
\|U\|^2 \sqrt{\log\left(\frac{2n}{\delta}\right)\sum_{\ell=1}^kv_\ell^4} \le \beta^2 \sqrt{\log\left(\frac{2n}{\delta}\right)\sum_{j=1}^k \frac{1}{k^2}} \le \frac{\varepsilon_0^2}{4}.
\end{equation*}
Thus, with probability at least $1-\delta$, \cite[Lemma 6.4, p. 21]{OymakFabianLiSoltanolkotabi:2019} guarantees that there exists a matrix $J\in\R^{n{\color{black} k}\times k}$ such that
$\Sigma(U) = J J^\top $ and $\|\mathcal{J}(C_0) - J\|\le \varepsilon_0.$
Therefore, the choice of $J$ fulfills Assumption \ref{a2} with probability at least $1-\delta$. Next, by \cite[Lemma 5]{HeckelSoltanolkotabi:2020denoising}, we have
\begin{equation*}
\|\mathcal{J}(C)\|\le \|v\| \|U\|\qquad\mbox{and}\qquad \|J\|\le \|v\|\|U\|,
\end{equation*}
and since $\|v\|=1$ and $\|U\|\le 1$, we deduce that also Assumption \ref{a1} is fulfilled. Last, we prove the validity of Assumption \ref{a3} using \cite[Lemma 7, p. 25]{HeckelSoltanolkotabi:2020denoising}, which states that for all $C$ which fulfill
$ \|C-C_0\|\le \omega \tilde{R}$ with $ \tilde{R}\le \frac{\sqrt{k}}{2}$,
with probability at least $1-n\exp (-\frac{1}{2}\tilde{R}^\frac{4}{3} k^\frac{1}{3})$, the following estimate holds
\begin{equation*}
\|\mathcal{J}(C) - \mathcal{J}(C_0)\|\le \|v\|_\infty 2(k\tilde{R})^\frac{1}{3}\|U\|\le \frac{1}{\sqrt{k}} 2 \left(k\left(\frac{\varepsilon}{4}\right)^3 \sqrt{k}\right)^\frac{1}{3} = \frac{\varepsilon}{2}.
\end{equation*}
Thus Assumption \ref{a3} is also fulfilled.  Finally, the preceding statements show that with probability at least
$1-\delta- n\exp(-\frac{\tilde{R}^\frac{4}{3}k^\frac{1}{3}}{2}) = 1-\delta -n\exp(-\frac{\varepsilon^4}{2^9}k)$,
Assumptions \ref{a1}--\ref{a3} are fulfilled.
\end{proof}

\begin{remark}
\cite[Lemma 7, p. 25]{HeckelSoltanolkotabi:2020denoising} states that the probability is at least $1-ne^{-\frac{1}{2}\tilde{R}^\frac{4}{3} k^\frac{7}{3}}$, which is incorrect. Indeed, in \cite[eq. (43)]{HeckelSoltanolkotabi:2020denoising}, Hoeffding's inequality was not applied correctly: the argument of the function $\exp(\cdot)$ should be $\frac{m^2}{2k}$ instead of $\frac{km^2}{2}$. However, this just changes the probability slightly and does not influence the overall correctness of the final result.
\end{remark}

\subsection{General error bound for the convolutional generator}\label{ssec:general-bound} We are now ready to state and prove an independent main result, given in Theorem \ref{s2:t2}. This result gives a general error bound on the approximation by the convolutional neural network generator $G(C)$. Recall that in Section \ref{ssec:conv-generator}, the reference Jacobian $J\in \mathbb{R}^{n{\color{black} k}\times k}$ is chosen such that
\begin{equation}\label{eqn:ref-Jacobian}
JJ^\top =\Sigma(U)=\E\left[\mathcal{J}(C)\mathcal{J}(C)^\top\right]\in\R^{n\times n}.
\end{equation}
It was shown in Heckel and Soltanolkotabi
\cite[Appendix H, p. 27] {HeckelSoltanolkotabi:2020denoising} that
\begin{equation*}
\Sigma(U)=\left(\frac{(u_i,u_j)}{2}\left(1-\frac{1}{\pi}\cos^{-1}\left(\frac{(u_i,u_j)}{\|u_i\|\|u_j\|}\right)\right)\right)_{ij}\in\R^{n\times n},
\end{equation*}
where $u_i$ denotes the $i$th row of the circulant matrix $U\in\R^{n\times n}$.
Hence $\Sigma(U)$ only depends on $U$. In particular, it is independent of $k$, the width of the convolutional neural network $G(C)$. We denote by $(\sigma_i^2,w_i)$ {\color{black} and $(\sigma_j'^2,w_j')$ the non-zero eigenvalues and orthonormal eigenvectors of the positive semi-definite  matrices $\Sigma(U)\in\mathbb{R}^{n\times n}$ and} $A\Sigma(U) A^\top\in\R^{m\times m}$, respectively:
\begin{align*}
    \Sigma(U) w_i &= \sigma_i^2 w_i,\quad i=1,\ldots,{\color{black} n'},\\
    A\Sigma(U)A^\top w_j' &= \sigma_j'^2w_j',\quad j=1,\ldots, {\color{black} m'},
\end{align*}
{\color{black} for some $n'\le n, m'\le m$.} Again we emphasise that they are independent of $k$. Accordingly, the (economic) singular value decompositions (SVDs) of $J$ and $AJ$ are given by $(\sigma_i,z_{i,k},w_i)$ and $(\sigma_j',z_{j,k}',w_j')$, i.e.,
\begin{align*}
    J z_{i,k} &= \sigma_i w_i ,\quad i=1,\ldots,n',\qquad \mbox{or } J = \sum_{i=1}^{n'}\sigma_i w_i z_{i,k}^\top,\\
    AJ z_{j,k}'  &= \sigma_j' w_j', \quad j =1,\ldots,m',\qquad \mbox{or }AJ = \sum_{j=1}^{m'} \sigma_j'w_j'z_{j,k}'^\top,
\end{align*}
where the right singular vectors $z_{i,k}$ and $z_{j,k}'$ can be defined respectively by
\begin{align}\label{eqn:svd}
    z_{i,k}:=J^\top w_i/\sigma_i\quad\mbox{and} \quad z_{j,k}':=J^\top A^\top w_j'/\sigma_j'.
\end{align}
Clearly the right singular vectors $(z_{i,k})_{i=1}^{n'}$ and $(z_{j,k}')_{j=1}^{m'}$ do depend on $k$, and hence they are explicitly indicated by the subscript $k$. Below we employ the SVDs of $J$ and $AJ$ frequently. Here the notation $P_{\mathcal{R}(J)}$ denotes the orthogonal projection into the range space $\mathcal{R}(J)$ of $J$ (so it differs slightly from the projection into the affine space $\mathcal{R}(\tilde G)$ of the linearized generator $\tilde G$, due to the presence of a nonzero $C_0$). The orthogonal projection $P_{\mathcal{R}(J)}$ is explicitly given by 
\begin{equation}\label{eqn:proj-J}
    P_{\mathcal{R}(J)} x = \sum_{i=1}^{n'} (w_i,x)w_i.
\end{equation}

\begin{theorem}\label{s2:t2}Let $\|A\| \le \gamma=1$, $x^\dag=A^\dag y$ be the minimum norm solution to problem \eqref{eqn:lin}, and $U\in\R^{n\times n}$ a fixed matrix with $\|U\|\le\beta=1$.
Let  $\eta=1$ be the constant step size and $4\le T \le \frac{1}{ 2\varepsilon^2}$ be the maximal iteration number. Fix $\delta>0$ and $0<\xi\le \frac{1}{\sqrt{32\log\left(\frac{2n}{\delta}\right)}}$, {\color{black}and set $\varepsilon:= \frac{\xi}{4T^2}$}. Consider a convolutional generator $G(C)$ as in \eqref{e0} with
\begin{equation}\label{eqn:k}
 k\ge 2^{35} \xi^{-8}n\log\left(\frac{2n}{\delta}\right)T^{13}.
\end{equation}
Let $(G(C_\tau^\epsilon))_{\tau}$ be the iterates of a gradient descent method with an initial weight matrix $C_0\in\R^{n\times k}$, following i.i.d. $\mathcal{N}(0,\omega^2)$, with the standard deviation
\begin{equation}\label{eqn:omega-t2}
    \omega:=\frac{\xi\|y^\epsilon\|}{2\sqrt{8n\log\left(\frac{2n}{\delta}\right)}}.
\end{equation}
Then, with probability at least $1-3\delta$, the following error bound holds for all $\tau\le T$,
\begin{align*}
\|x^\dag-G(C_\tau^\epsilon)\|&\le \sqrt{\sum_{i=1}^{n'} \sigma_i^2 \left(\sum_{j=1}^{m'} \frac{1-(1-\eta\sigma_j'^2)^\tau}{\sigma_j'}(w_j',y^\epsilon-y)(z_{j,k}',z_{i,k})\right)^2 }\\
&\qquad + \sqrt{\sum_{i=1}^{n'}\left(\sigma_i \sum_{j=1}^{m'}\frac{1-(1-\sigma_j'^2)^\tau}{\sigma_j'}(A^\top w_j',x^\dag)(z_{j,k}',z_{i,k}) - (x^\dag,w_i)\right)^2}\\
&\qquad + \sqrt{\sum_{i=1}^{n'}\sigma_i^2\left(\sum_{j=1}^{m'} \left(\frac{1-(1-\sigma_j'^2)^\tau}{\sigma_j'}(A^\top w_j',G(C_0))\right) (z_{j,k}',z_{i,k})\right)^2}\\
&\qquad + 3\xi\|y^\epsilon\|+ \|(I-P_{\mathcal{R}(J)})x^\dag\|.
\end{align*}
\end{theorem}
 \begin{proof} 
The proof is lengthy, and we divide it into four steps. We define
\begin{align}\label{eqn:para-choice}
 \varepsilon_0:=\min\left(\left(\frac{16\log\left(\frac{2n}{\delta}\right)}{k}\right)^\frac{1}{4}, \varepsilon\right)\quad \mbox{and}\quad
 R:=8\sqrt{T} \|y^\epsilon\|,
 \end{align}
and then apply Proposition \ref{s2:p1} and Corollary \ref{s2:c1}.

\medskip
\noindent \textbf{Step 1: Verify the condition $R\le \omega \left(\frac{\varepsilon}{4}\right)^3 \sqrt{k}$.} This is one key condition for applying Proposition \ref{s2:p1}. Indeed, by the definitions of $R$ and $\omega$ and the choice of $k$ in \eqref{eqn:k}, we have
\begin{align*}
\frac{R}{\omega \sqrt{k}}&\!\!=\!\! \frac{16\sqrt{8n\log\left(\frac{2n}{\delta}\right)}\sqrt{T}}{\xi} \frac{1}{\sqrt{k}}
\leq \frac{2^4\sqrt{8n\log\left(\frac{2n}{\delta}\right)}\sqrt{T}}{\xi} \frac{\xi^4}{2^{16}\sqrt{8n\log\left(\frac{2n}{\delta}\right)}T^{\frac{13}{2}}}\!\\
&=\!\frac{1}{2^6}\left(\frac{\xi}{4T^2}\right)^3\!\!\!= \Big(\frac{\varepsilon}{4}\Big)^3,
\end{align*}
where the last step is due to the definition of $\varepsilon$.
Moreover, since $0<\xi<1$, by the choice of the neural network width $k$ and the definition of $\varepsilon$,
\begin{align*}
k\ge \frac{2^{35}n\log\left(\frac{2n}{\delta}\right)T^{13}}{\xi^8} &\ge \frac{2^{11}\log\left(\frac{2n}{\delta}\right)T^8}{\xi^4}
= \frac{ 2^3\log\left(\frac{2n}{\delta}\right) 4^4 T^8}{\xi^4}
=\frac{ 2^3\log\left(\frac{2n}{\delta}\right)}{\varepsilon^4}\ge\frac{ 2^3\log\left(\frac{2n}{\delta}\right)}{\varepsilon_0^4}.
\end{align*}
Thus, by Proposition \ref{s2:p1}, the choices of $k$, $\omega$ and $\eta$ ensure that Assumptions \ref{a1}, \ref{a2} and \ref{a3} are fulfilled with probability at least
$1-\delta -n\exp\left(-\frac{\varepsilon^4}{2^9}k\right) \ge 1-2\delta$.\\
\noindent \textbf{Step 2: Verifying the assumptions of Theorem \ref{s2:t1}.}
To employ Theorem \ref{s2:t1} (or Corollary \ref{s2:c1}), next we verify the above specific choices for the assumptions of Theorem \ref{s2:t1}, especially the condition that the chosen $R$ is indeed sufficiently large. To this end, we first establish a bound on the initial residual $\|r_0\|$ and on the initial estimate $\|G(C_0)\|$. Actually, \cite[Lemma 6, p. 25]{HeckelSoltanolkotabi:2020denoising} states that with probability at least $1-\delta$,
\begin{equation*}
\|G(C_0)\| \le \omega \sqrt{8 \log\left(\frac{2n}{\delta}\right)}\|U\|_F,
\end{equation*}
and by the choice of $\omega$ in \eqref{eqn:omega-t2} and using the elementary inequality $\|U\|_F\le \sqrt{n}\|U\|$ and the condition $\|U\|=1$, we deduce
\begin{equation}\label{p1:e3}
\|G(C_0)\| \le\omega\sqrt{8n\log\left(\frac{2n}{\delta}\right)} \le \xi \|y^\epsilon\|.
\end{equation}
This, the triangle inequality and the condition $\|A\|=1$ yield
\begin{equation*}
\|r_0^\epsilon\|\le \|AG(C_0) - y^\epsilon\|\le \|y^\epsilon\| + \|A\|\|G(C_0)\|\le \left(1+\xi\right)\|y^\epsilon\|\le 2\|y^\epsilon\|,
\end{equation*}
with probability at least $1-\delta$. Therefore, by the choice of $\xi$ and $T$, we have
\begin{align*}
 2\|r_0^\epsilon\|(\sqrt{T}+2\varepsilon T^2) = 2\|r_0^\epsilon\|(\sqrt{T}+\tfrac{\xi}{2}) \leq 4\|r_0^\epsilon\|\sqrt{T}\le 8\|y^\epsilon\|\sqrt{T},
\end{align*}
 with probability at least $1-\delta$ and consequently,
$R=8\sqrt{T}\|y^\epsilon\|\ge 2\|r_0^\epsilon\|(\sqrt{T}+2\varepsilon T^2)$.
By the definition of $\varepsilon$,
$T\le \frac{1}{2\varepsilon^2}$
and thus the assumptions of Theorem \ref{s2:t1} (for $y^\epsilon$) are fulfilled with probability at least $1-3\delta$. Thus, we may employ the estimates in Corollary \ref{s2:c1}.

\medskip
\noindent\textbf{Step 3: Bound the errors of linearized and nonlinear iterates.} We denote by $C^\epsilon_\tau$ and $\tilde{C}^\epsilon_\tau$ the iterates for noisy data $y^\epsilon$, and by $C_\tau$ and $\tilde{C}_\tau$ the ones for exact data $y$. Similarly, we denote by $r_\tau^\epsilon$ and $\tilde{r}_\tau^\epsilon$ the residuals for noisy data $y^\epsilon$, and by $r_\tau$ and $\tilde{r}_\tau$ the corresponding residuals for the exact data $y$. For the linearized neural network $\tilde{G}(C)$, we write
$\tilde{G}(C):=G(C_0) + J (c-c_0)$, where $c\in\R^{nk}$ is the vectorized version of $C\in\R^{n\times k}$ (and similarly for $c_0$ and $c_\tau^\epsilon$ etc.), and $J$ is the reference Jacobian defined in \eqref{eqn:ref-Jacobian}. First, we decompose the nonlinear estimation error $\|G(C^\epsilon_\tau)-x^\dag\|$ into
\begin{align*}
\|G(C^\epsilon_\tau)-x^\dag\|&\le \|G(C^\epsilon_\tau)-\tilde{G}(\tilde{C}^\epsilon_\tau)\| + \|\tilde{G}(\tilde{C}^\epsilon_\tau) - x^\dag\|.
\end{align*}
For the first term, we have
\begin{align*}
& \|G(C^\epsilon_\tau) - \tilde{G}(\tilde{C}^\epsilon_\tau)\|\le \|G(C^\epsilon_\tau) - \tilde{G}(C^\epsilon_\tau)\| + \|\tilde{G}(C^\epsilon_\tau) - \tilde{G}(\tilde{C}^\epsilon_\tau)\|\\
=&\|G(C^\epsilon_\tau) - \left(G(C_0) + J(c^\epsilon_\tau - c_0)\right)\| + \|G(C_0)+J(c_\tau^\epsilon-c_0) -(G(C_0) + J(\tilde{c}^\epsilon_\tau-c_0)\|\\
=&\|G(C^\epsilon_\tau) - Jc_\tau^\epsilon - \left(G(C_0) - Jc_0\right)\| + \|J(c_\tau^\epsilon - \tilde{c}^\epsilon_\tau)\|.
\end{align*}
Note that the Jacobian of the mapping $C\mapsto G(C) - Jc$ is given by $\mathcal{J}(C)-J$. By the intermediate value theorem, Assumption \ref{a3} and Corollary \ref{s2:c1}, we have the bound 
\begin{align*}
\|G(C^\epsilon_\tau) - J c_\tau^\epsilon - \left(G(C_0) - Jc_0\right)\|
 \le \sup_{\xi\in\overline{\rm conv}(C_0,C^\epsilon_\tau)}\|\mathcal{J}(\xi)-J\|\|c_\tau^\epsilon-c_0\|\le \varepsilon \tfrac{R}{2}.
\end{align*}
For the term $\|J(c_\tau^\epsilon - \tilde{c}^\epsilon_\tau)\|$, using Corollary \ref{s2:c1} again, we have
\begin{equation*}
\|J(c_\tau^\epsilon - \tilde{c}^\epsilon_\tau)\| \le \|J\|\|c^\epsilon_\tau-\tilde{c}^\epsilon_\tau\|\le  2\varepsilon \tau^2 \|r_0^\epsilon\|.
\end{equation*}
By the choice of the parameter $R$ in \eqref{eqn:para-choice} and the definition of and $\varepsilon$  (i.e., $4\leq T$ and $\varepsilon=\frac{\xi}{4T^2}$), we obtain
\begin{align*}
\varepsilon R= \varepsilon 8\sqrt{ T}\|y^\epsilon\| = \frac{2\xi}{T^\frac{3}{2}} \|y^\epsilon\|\le \xi \|y^\epsilon\|\quad \mbox{and}\quad
2\varepsilon \tau^2 \|r_0^\epsilon\|\le 4\varepsilon T^2 \|y^\epsilon\|\le \xi \|y^\epsilon\|.
\end{align*}
Hence, we obtain 
$\|G(C_\tau^\epsilon)-\tilde G(\tilde C_\tau^\epsilon)\|\leq \tfrac{3}{2}\xi\|y^\epsilon\|$.

\medskip
\noindent\textbf{Step 4: Bound the linear approximation error.}
Now, for the linear approximation error $\tilde{G}(\tilde{C}^\epsilon_\tau) - x^\dag$, we decompose it further into three terms
\begin{align*}
&\|\tilde{G}(\tilde{C}^\epsilon_\tau) - x^\dag\|\le \|\tilde{G}(\tilde{C}^\epsilon_\tau) - \tilde{G}(\tilde{C}_\tau)\| + \|\tilde{G}(\tilde{C}_\tau) - x^\dag\|\\
\leq&\|\tilde{G}(\tilde{C}^\epsilon_\tau) - \tilde{G}(\tilde{C}_\tau)\| + \|\tilde{G}(\tilde{C}_\tau) - P_{\mathcal{R}(J)}x^\dag\|+\|(I-P_{\mathcal{R}(J)})x^\dag\|.
\end{align*}
The three terms represent the noise propagation error (for the linearized model, due to the presence of data noise), the approximation error (depending on $\tau$) and the projection error (due to the limited expressivity of the representation using $\tilde G$ or $J$).
Next we estimate the first two terms using the SVDs of $J$ and $AJ$. From the representations of the linear approximations $\tilde c_\tau$ and $\tilde c_\tau^\epsilon$ and the SVD of $AJ$
\begin{align*}
\tilde{c}_\tau &= c_0 -\sum_{j=1}^{m'} \frac{1-(1-\sigma_j'^2)^\tau}{\sigma_j'}(r_0,w_j')z_{j,k}'\quad\mbox{and}\quad
\tilde{c}_\tau^\epsilon = c_0 -\sum_{j=1}^{m'} \frac{1-(1-\sigma_j'^2)^\tau}{\sigma_j'}(r_0^\epsilon,w_j')z_{j,k}',
\end{align*}
cf. \eqref{eqn:theta-tilde}, and the SVDs of $AJ$ and $J$, it follows that the noise propagation error
$\tilde{G}(\tilde{C}^\epsilon_\tau) - \tilde{G}(\tilde{C}_\tau)$ can be represented by
\begin{align*}
\tilde{G}(\tilde{C}^\epsilon_\tau) - \tilde{G}(\tilde{C}_\tau)&=-\sum_{j=1}^{m'}\frac{1-(1-\eta\sigma_j'^2)^\tau}{\sigma_j'}(w_j',r^\epsilon_0-r_0) J z_{j,k}' \\
&=-\sum_{j=1}^{m'}\frac{1-(1-\eta\sigma_j'^2)^\tau}{\sigma_j'}(w_j',y^\epsilon-y) \sum_{i=1}^{n'} (z_{j,k}',z_{i,k}) \sigma_i w_i\\
&=-\sum_{i=1}^{n'} \sigma_i \left(\sum_{j=1}^{m'} \frac{1-(1-\eta\sigma_j'^2)^\tau}{\sigma_j'}(w_j',y^\epsilon-y)(z_{j,k}',z_{i,k})\right) w_i.
\end{align*}
Last, we bound the linearized approximation error $\tilde{G}(\tilde{C}_\tau)-\mathcal{P}_{\mathcal{R}(J)}x$. Using the defining relations $r_0=AG(C_0)-y$ and $\tilde G(\tilde C_\tau)= G(C_0)+J(\tilde c_\tau-c_0)$, the SVDs of $AJ$ and $J$, and the expression for $\mathcal{P}_{\mathcal{R}(J)}$ in \eqref{eqn:proj-J}, we derive
\begin{align*}
\tilde{G}(\tilde{C}_\tau)-\mathcal{P}_{\mathcal{R}(\tilde{G})}x^\dag =&G(C_0)-\sum_{j=1}^{m'} \frac{1-(1-\sigma_j'^2)^\tau}{\sigma_j'}(w_j',r_0)Jz_{j,k}' - \mathcal{P}_{\mathcal{R}(J)}x^\dag\\
=&G(C_0) -\sum_{j=1}^{m'} \frac{1-(1-\sigma_j'^2)^\tau}{\sigma_j'}(A^\top w_j',G(C_0))Jz_{j,k}' \\
  &+ \sum_{j=1}^{m'} \frac{1-(1-\sigma_j'^2)^\tau}{\sigma_j'}(w_j',y)Jz_{j,k}' - \sum_{i=1}^{n'} (x^\dag,w_i) w_i.
\end{align*}
Now substituting the SVD of $J$ leads to
\begin{align*}
\tilde{G}(\tilde{C}_\tau)-\mathcal{P}_{\mathcal{R}(J)}x^\dag=&G(C_0)-\sum_{i=1}^{n'}\sigma_i\left(\sum_{j=1}^{m'} \left( \frac{1-(1-\sigma_j')^\tau}{\sigma_j'}(A^\top w_j',G(C_0))\right) (z_{j,k}',z_{i,k})\right) w_i\\
&+ \sum_{i=1}^{n'}\left(\sigma_i \sum_{j=1}^{m'}\frac{1-(1-\sigma_j'^2)^\tau}{\sigma_j'}(A^\top w_j',x^\dag)(z_{j,k}',z_{i,k}) - (x^\dag,w_i)\right) w_i .
\end{align*}
Combining the preceding identities with the orthonormality of the singular vectors $w_i$ completes the proof of the theorem.
\end{proof}

\begin{remark}
Theorem \ref{s2:t2} indicates that the representation via $J$ or $G$ has a major impact on the error analysis, through the interaction terms $(z_{j,k}',z_{i,k})$. Undoubtedly, this greatly complicates the error analysis of the synthesis approach in the general case, and a suitable alignment between the right singular vectors $z_{j,k}'$ and  $z_{i,k}$ is needed in order to simplify the analysis. This strategy is followed in Theorem \ref{t0} and will be discussed in Section \ref{s3}. 
\end{remark}

\section{The proof of Theorem \ref{t0}}\label{s3}
Now we give the proof of Theorem \ref{t0}. There is a major difference to the classical analysis from the inverse problems literature \cite{EnglHankeNeubauer:1996} due to the uncommon representation of $x$ in another basis $J$ or $G$, and the presence of the random initial iterate $C_0$ and $G(C_0)$. Nonetheless, the discrepancy principle \eqref{eqn:dp} can still be analyzed classically and is optimal uniformly over the source condition in the sense of worst-case-error.
\begin{proof}[Proof of Theorem \ref{t0}]
The proof uses Theorem \ref{s2:t2} with the parameter setting
\begin{align*}
\xi_\epsilon&=\min\left(\frac{(L-1)\epsilon}{4\|y^\epsilon\|},\frac{1}{\sqrt{32\log\left(\frac{2n}{\delta_\epsilon}\right)}}\right),\\
 T_\epsilon&\ge\max\left(\left(\frac{2\rho}{(L-1)\epsilon}\right)^{\frac{2(p+q)}{q(1+\nu)}}\frac{q(1+\nu)C_A^\frac{q+p}{q}}{2e(p+q)c_A c_\Sigma}, \frac{4 \|y^\epsilon\|}{(L-1)\epsilon}\right) \quad\mbox{and}\quad \varepsilon_\epsilon = \frac{\xi_\epsilon}{4T_\epsilon^2}.
\end{align*}
We divide the lengthy proof into five steps, each bounding the different sources of the error. For simplicity we also assume $m=n$ and $\mathcal{N}(A^\top A)=\mathcal{N}(\Sigma(U)) = \{0\}$.
\medskip

\noindent{}{\bf Step 1: Show $(z_{j,k}',z_{i,k})=\delta_{ij}$ under the given spectral alignment assumption.} We first show that the condition on $A$ and $\Sigma(U)$, i.e., $\Sigma(U)$ and $A^\top A$ share the (orthonormal) eigenbasis $(w_i)_{i=1}^n$, imply
\begin{equation}\label{eqn:ortho}
(z_{j,k}',z_{i,k}) = \delta_{ij}.
\end{equation}
Indeed, with $u_i:=Aw_i/\alpha_i$, we have
\begin{equation}\label{t0:e1}
A \Sigma(U) A^\top u_i = \frac{1}{\alpha_i} A \Sigma(U) A^\top A w_i = \alpha_i A \Sigma(U) w_i = \alpha_i \sigma_i^2 A w_i = \alpha_i^2 \sigma_i^2 u_i,
\end{equation}
which implies $w_i' = u_i$. Therefore, by the defining relations for $z_{j,k}'$ and $z_{i,k}$ in \eqref{eqn:svd}, we have
\begin{align*}
(z_{j,k}',z_{i,k}) &= \frac{1}{\sigma_j' \sigma_i}\left( J^\top A^\top w_j',J^Tw_i\right) = \frac{1}{\sigma_j'\sigma_i} \left( J^\top A^\top u_j,J^\top w_i\right).
\end{align*}
This, the identity $ A^\top u_j = A^\top A w_j/\alpha_j = \alpha_jw_j$ and the relation $\Sigma(U)=JJ^\top$ in \eqref{eqn:ref-Jacobian} imply
\begin{align*}
(z_{j,k}',z_{i,k})&= \frac{\alpha_j}{\sigma_j'\sigma_i} \left( J^\top w_j, J^\top w_i\right) = \frac{\alpha_j}{\sigma_j'\sigma_i} \left(   w_j, JJ^\top w_i\right)\\
&=\frac{\alpha_j}{\sigma_j'\sigma_j}\left(w_j, \Sigma(U) w_i\right) = \frac{\alpha_j\sigma_i}{\sigma_j'}\left( w_j,w_i\right)=\frac{\alpha_j\sigma_i}{\sigma_j'}\delta_{ij}=\delta_{ij},
\end{align*}
since $\sigma_j'^2$ are the eigenvalues of $A\Sigma(U)A^\top $, which were proven to be equal to $\sigma_j^2\alpha_j^2$ in \eqref{t0:e1}. The identity \eqref{eqn:ortho} allows simplifying the double summation in Theorem \ref{s2:t2}. Moreover, {\color{black} since all the eigenvalues of $\Sigma(U)=JJ^\top$ are positive by assumption, it follows that $\{0\}=\mathcal{N}(J^\top) = \mathcal{R}(J)^\perp$, which implies $\R^n=\mathcal{R}(J) = G(C_0)+\mathcal{R}(J)=\mathcal{R}(\tilde{G})$ and thus} $\|\mathcal{P}_{\mathcal{R}(\tilde{G})}  x^\dag -x^\dag \|=0$.
With this fact and the identity \eqref{eqn:ortho}, Theorem \ref{s2:t2} gives the following error bound 
\begin{align*}
\|G(C^\epsilon_{\tau^\epsilon_{\rm dp}}) - x^\dag\|\le& \sqrt{\sum_{i=1}^n \sigma_i^2 \frac{(1-(1-\eta\sigma_i'^2)^{\tau^\epsilon_{\rm dp}})^2}{\sigma_i'^2}(w_i',y^\epsilon-y)^2}\\
& + \sqrt{\sum_{i=1}^n\left(\sigma_i \frac{1-(1-\sigma_i'^2)^{\tau^\epsilon_{\rm dp}}}{\sigma_i'^2}(A^\top w_i',x^\dag) - (x^\dag,w_i)\right)^2}\\
& + \sqrt{\sum_{i=1}^n\sigma_i^2\frac{(1-(1-\sigma_i')^{\tau^\epsilon_{\rm dp}})^2}{\sigma_i'^2}(A^\top w_i',G(C_0))^2} + 3\xi_\epsilon\|y^\epsilon\|
=: \sum_{i=1}^3\mathcal{E}_i + 3\xi_\epsilon\|y^\epsilon\|.
\end{align*}
The error $\mathcal{E}_1$ arises from the presence of data noise (and hence called data noise propagation error), the error $\mathcal{E}_2$ arises from the early stopping of the iteration and depends on the source condition \eqref{eqn:source} (and hence approximation error), and the last term $\mathcal{E}_3$ arises from the nonzero initial condition $C_0$ in the convolutional generator $G(C)$. 
The remaining task is to bound these three terms $\mathcal{E}_i$, $i=1,2,3$, below.

\medskip
\noindent{\bf Step 2: Bounding the error induced by the nonzero  initial point $C_0$}. Now we analyze the error $\mathcal{E}_3$ due to nonzero initial point $C_0$, which is essentially independent of the stopping index. 
Indeed, since $\sigma_i'=\alpha_i\sigma_i$ and $A^Tw_i'=\alpha_i w_i$, for any $\tau\in\mathbb{N}$, by the SVD of $A$, we have
\begin{align}
\mathcal{E}_3= &\sqrt{\sum_{i=1}^n\sigma_i^2\left(\frac{1-(1-\alpha_i^2\sigma_i^2)^{\tau^\epsilon_{\rm dp}}}{\alpha_i\sigma_i} \alpha_i(w_i,G(C_0))\right)^2}\nonumber\\
=&\sqrt{\sum_{i=1}^n \left(1-(1-\alpha_i^2\sigma_i^2)^{\tau^\epsilon_{\rm dp}}\right)^2(w_i,G(C_0))^2}
\le\|G(C_0)\|.\label{eqn:basic-est-GC_0}
\end{align}
Moreover, with a probability at least $1-\delta_\epsilon$, by \eqref{p1:e3}, we have
$\|G(C_0)\| \le \xi_\epsilon \|y^\epsilon\|$,
and consequently, together with \eqref{eqn:basic-est-GC_0} and the choice of $\xi_\epsilon$, we obtain
$\mathcal{E}_3\le \xi_\epsilon\|y^\epsilon\| \le \tfrac{L-1}{4}\epsilon$.

\medskip
\noindent{\bf Step 3: Relation between nonlinear and linearized residuals.} Now we analyze the discrepancy principle \eqref{eqn:dp}. First, we show that the stopping index $\tau_{\rm dp}^\epsilon$ is well-defined. Indeed, by the triangle inequality and the expression \eqref{eqn:res-lin} of the linearized residual $\tilde r_{T_\epsilon}$, we have
 \begin{align*}
 &\|r^\epsilon_{T_\epsilon}\|\le \|\tilde{r}_{T_\epsilon}\| +\|\tilde{r}^\epsilon_{T_\epsilon}-\tilde{r}_{T_\epsilon}\|+\|r_{T_\epsilon}^\epsilon-\tilde{r}^\epsilon_{T_\epsilon}\|\\
 =& \|(I-AJJ^\top A^\top)^{T_\epsilon}(AG(C_0) - y)\|+\|\tilde{r}^\epsilon_{T_\epsilon}-\tilde{r}_{T_\epsilon}\|+\|r^\epsilon_{T_\epsilon}-\tilde{r}^\epsilon_{T_\epsilon}\|\\
 \leq&\|(I-AJJ^\top A^\top)^{T_\epsilon}y\|+\|(I-AJJ^\top A^\top)^{T_\epsilon}AG(C_0)\|
 +\|\tilde{r}^\epsilon_{T_\epsilon}-\tilde{r}_{T_\epsilon}\|+\|r^\epsilon_{T_\epsilon}-\tilde{r}^\epsilon_{T_\epsilon}\|:= \sum_{\ell=1}^4 {\rm I}_\ell. 
\end{align*}
Next we bound the four terms ${\rm I}_\ell$, $\ell=1,\ldots,4$, separately.
From the source condition $x^\dag\in \mathcal{X}_{\nu,\rho}$ in  \eqref{eqn:source}, i.e., $x^\dag = (A^\top A)^\frac{\nu}{2}v$, and the SVDs of $J$ and $AJ$ we deduce
\begin{align}
   {\rm I}_1  & = \sqrt{\sum_{i=1}^n(1-\alpha_i^2\sigma_i^2)^{2T_\epsilon}(y,w_i')^2}=\sqrt{\sum_{i=1}^n(1-\alpha_i^2\sigma_i^2)^{2T_\epsilon} \alpha_i^{2(1+\nu)}(v,w_i)^2}\nonumber\\
   &\leq \rho \max_{1\le i \leq n}(1- \sigma_i^2\alpha_i^2)^{T_\epsilon}\alpha_i^{1+\nu}.\label{eq:approx}
\end{align}
Now by the assumption on the decay of $\alpha_i^2$ and $\sigma_i^2$,
\begin{align*}
\alpha_i^2&\le B_A i^{-q} = B_A \left( i^{-(q+p)}\right)^\frac{q}{q+p} \le B_A \left(\frac{\alpha_i^2\sigma_i^2}{b_Ab_\Sigma}\right)^\frac{q}{q+p} = \frac{B_A}{\left(b_A b_\Sigma\right)^\frac{q}{q+p}} \left(\alpha_i^2\sigma_i^2\right)^\frac{q}{q+p},
\end{align*}
and since $\|A\|\|\Sigma(U)\|\le 1$ by assumption, we have
\begin{align}\label{eq:approx:2}
   {\rm I}_1 &\leq \left(\frac{B_A}{\left(b_A b_\Sigma\right)^\frac{q}{q+p}}\right)^\frac{1+\nu}{2} \rho \sup_{0<\lambda\leq 1}(1- \lambda)^{T_\epsilon}\lambda^{\frac{q(1+\nu)}{2(q+p)}}\le\left(\frac{q(1+\nu) B_A^\frac{q+p}{q}}{2e(q+p)b_Ab_\Sigma}
    \right)^\frac{q(1+\nu)}{2(q+p)} \rho T_\varepsilon^{-\frac{q(1+\nu)}{2(q+p)}},
\end{align}
where the last step follows from Lemma \ref{lem:basic-est2} in the appendix (with the exponent $r=\frac{q(1+\nu)}{2(q+p)}$). 
For the term ${\rm I}_2$, by the contraction property $\|(I-AJJ^\top A^\top)\|\leq 1$, the condition $\|A\|=1$ and the estimate \eqref{p1:e3}, we have with a probability at least $1-\delta_\epsilon$, 
\begin{equation*}
        {\rm I}_2 = \|(I-AJJ^\top A^\top)^{T_\epsilon}AG_0\| \leq \|G(C_0)\| \leq \xi_\epsilon \|y^\epsilon\|.
\end{equation*}
Likewise, by Corollary \ref{s2:c1}, we have
\begin{align*}
    {\rm I}_3 &= \sqrt{\sum_{i=1}^n(1-\alpha_i^2\sigma_i^2)^{2T_\epsilon}(r^\epsilon_0-r_0,w_i')^2} \leq \|y^\epsilon-y\| = \epsilon,\\
    {\rm I}_4 &= \|r^\epsilon_{T_\epsilon} - \tilde{r}^\epsilon_{T_\epsilon}\|\leq 4\varepsilon_\epsilon T_\epsilon \|r_0^\epsilon\|.
\end{align*}
Hence, by the choices of $T_\epsilon$ (i.e., $T_\epsilon\geq (\frac{2\rho}{(L-1)\epsilon})^{\frac{2(p+q)}{q(1+\nu)}}\frac{q(1+\nu)B_A^\frac{q+p}{q}}{2e(p+q)b_A b_\Sigma}$), $\xi_\epsilon$ and $\varepsilon_\epsilon$, we obtain with a probability at least $1-\delta_\epsilon$,
 \begin{align*}
  \|r^\epsilon_{T_\epsilon}\|
   &\le \epsilon +  \left(\frac{q(1+\nu)B_A^\frac{q+p}{q}}{2e(p+q)b_Ab_\Sigma}\right)^\frac{q(1+\nu)}{2(q+p)}\!\! \!\rho T_\epsilon^{-\frac{q(1+\nu)}{2(p+q)}} + \xi_\epsilon\|y^\epsilon\|\Big(1+\frac{2}{T_\epsilon}\Big)
   \le \epsilon +\frac{L-1}{2}\epsilon + \frac{L-1}{2}\epsilon = L\epsilon.
 \end{align*}
Thus, the stopping index $\tau^\epsilon_{\rm dp}\le T_\epsilon$ is well-defined (for the discrepancy principle \eqref{eqn:dp}). Now we relate the stopped (nonlinear) residual $r^\epsilon_{\tau_{\rm dp}^\epsilon}$ to the linearized stopped residual $\tilde{r}^\epsilon_{\tau^\epsilon_{\rm dp}}$. Actually, by Corollary \ref{s2:c1} and the choice of $T_\epsilon$, with probability at least $1-\delta_\epsilon$, we have
\begin{align*}
  \|\tilde{r}^\epsilon_{\tau^\epsilon_{\rm dp}} - r^\epsilon_{\tau^\epsilon_{\rm dp}}\| \leq
  4\varepsilon_\epsilon \tau_{\rm dp}^\epsilon \|r_0^\epsilon\|\le \frac{\xi_\epsilon\|y^\epsilon\|}{T_\epsilon} \leq \frac{L-1}{2}\epsilon.
\end{align*}
Consequently, we have the following lower and upper bounds:
 \begin{align}\label{t0:e2}
\|\tilde{r}^\epsilon_{\tau^\epsilon_{\rm dp}}\|&\le \|r_{\tau^\epsilon_{\rm dp}}\| + \|\tilde{r}^\epsilon_{\tau^\epsilon_{\rm dp}} - r^\epsilon_{\tau^\epsilon_{\rm dp}}\|\le L\epsilon + \frac{L-1}{2}\epsilon \le \frac{3L-1}{2}\epsilon,\\
\label{t0:e3}
\|\tilde{r}^\epsilon_{\tau^\epsilon_{\rm dp}-1}\|&\ge \|r_{\tau^\epsilon_{\rm dp}-1}\| - \|\tilde{r}^\epsilon_{\tau^\epsilon_{\rm dp}-1} - r^\epsilon_{\tau^\epsilon_{\rm dp}-1}\|\ge L\epsilon -  \frac{L-1}{2}\epsilon\ge \frac{L+1}{2}\epsilon.
 \end{align}
{\bf Step 4: Analysis of the stopped linearized iterates}. Now we can bound the remaining two terms $\mathcal{E}_1$ and $\mathcal{E}_2$, the data propagation and (linear) approximation errors:
 \begin{align*}
 \mathcal{E}_1&= \sqrt{\sum_{i=1}^n \frac{(1-(1-\alpha_i^2\sigma_i^2)^{\tau^\epsilon_{\rm dp}})^2}{\alpha_i^2}(w_i',y^\epsilon-y)^2},\\
\mathcal{E}_2&= \sqrt{\sum_{i=1}^n \left((1-(1-\alpha_i^2\sigma_i^2)^{\tau^\epsilon_{\rm dp}})(w_i,x^\dag) - (x^\dag,w_i)\right)^2}
=\sqrt{\sum_{i=1}^n(1-\alpha_i^2\sigma_i^2)^{2\tau^\epsilon_{\rm dp}}\alpha_i^{2\nu}(v,w_i)^2},
\end{align*}
under the source condition \eqref{eqn:source}.
First we consider the (linear) approximation error $\mathcal{E}_2$. By H\"{o}lder's inequality
$$\sum_{i=1}^n |x_i y_i| \le \left(\sum_{i=1}^n |x_i|^a\right)^\frac{1}{a} \left(\sum_{i=1}^n |y_i|^b\right)^\frac{1}{b}$$
for $a,b\geq 1$ with $\frac{1}{a}+\frac{1}{b} = 1$, with the exponents $a=\frac{\nu+1}{\nu}$ and $b=\nu+1$, we deduce
\begin{align*}
\sum_{i=1}^n(1-\alpha_i^2\sigma_i^2)^{2\tau^\epsilon_{\rm dp}} \alpha_i^{2\nu}(v,w_i)^2 &= \sum_{i=1}^n(1-\alpha_i^2\sigma_i^2)^{2\tau^\epsilon_{\rm dp}} \alpha_i^{2\nu} (v,w_i)^\frac{2\nu}{\nu+1} (v,w_i)^\frac{2}{\nu+1}\\
&\le \left(\sum_{i=1}^n(1-\alpha_i^2\sigma_i^2)^\frac{\tau^\epsilon_{\rm dp} (\nu+1)}{\nu} \alpha_i^{2(\nu+1)} (v,w_i)^2\right)^\frac{\nu}{\nu+1}\left(\sum_{i=1}^n(v,w_i)^2\right)^\frac{1}{\nu+1}.
\end{align*}
Furthermore, by the source condition $x^\dag \in \mathcal{X}_{\nu,\rho}$, we have
\begin{align*}
&\sum_{i=1}^n(1-\alpha_i^2\sigma_i^2)^\frac{\tau^\epsilon_{\rm dp} (\nu+1)}{\nu} \alpha_i^{2(\nu+1)} (v,w_i)^2\\
\le&\sum_{i=1}^n(1-\alpha_i^2\sigma_i^2)^{\tau^\epsilon_{\rm dp}} \alpha_i^{2(\nu+1)} (v,w_i)^2 
=\|(I-AJJ^\top A)^{\tau_{\rm dp}^\epsilon}y\|^2.
\end{align*}
Note that the quantity $(I-AJJ^\top A)^{\tau_{\rm dp}^\epsilon}y$ differs from the linearized residual $\tilde r_{\tau_{\rm dp}^\epsilon}^\epsilon:=A\tilde G(C_{\tau^\epsilon_{\rm dp}})-y=(I-AJJ^\top A^\top)^{\tau_{\rm dp}^\epsilon}(AG(C_0)-y)$ (for exact data $y$) defined in \eqref{eqn:res-lin} by the term involving $G(C_0)$. Next, by the triangle inequality, we have
\begin{align*}
\|(I-AJJ^\top A)^{\tau_{\rm dp}^\epsilon}y\|\le& \|(I-AJJ^\top A)^{\tau_{\rm dp}^\epsilon}(y-AG(C_0))\| + \|(I-AJJ^\top A)^{\tau_{\rm dp}^\epsilon}AG(C_0)\| \\
= & \|A\tilde G(\tilde C_{\tau^\epsilon_{\rm dp}})-y\| + \|(I-AJJ^\top A)^{\tau_{\rm dp}^\epsilon}AG(C_0)\|\\
\leq& \|AG(C_0)\| + \|A\tilde{G}(\tilde{C}^\epsilon_{\tau^\epsilon_{\rm dp}}) - y^\epsilon\| + \|A\tilde{G}(\tilde{C}_{\tau^\epsilon_{\rm dp}})- A\tilde{G}(\tilde{C}^\epsilon_{\tau^\epsilon_{\rm dp}}) + y^\epsilon-y\|\\
=&\|AG(C_0)\| + \|\tilde{r}^\epsilon_{\tau^\epsilon_{\rm dp}}\| + \|\tilde{r}_{\tau^\epsilon_{\rm dp}} - \tilde{r}^\epsilon_{\tau^\epsilon_{\rm dp}} \|.
\end{align*}
From the estimate \eqref{t0:e2}, we obtain
$\|\tilde{r}^\epsilon_{\tau^\epsilon_{\rm dp}}\|\le \frac{3L-1}{2} \epsilon$.
Moreover, since $\tilde{r}_\tau=(I-AJJ^\top A^\top)^\tau r_0$ and $\tilde{r}^\epsilon_\tau = (I-AJJ^\top A^\top)^\tau r^\epsilon_0$, we deduce
\begin{align*}
\|\tilde{r}_{\tau^\epsilon_{\rm dp}} - \tilde{r}^\epsilon_{\tau^\epsilon_{\rm dp}}\| = \|(I- A J J^\top A^\top)^{\tau^\epsilon_{\rm dp}}(y^\epsilon-y)\| \le \epsilon.
\end{align*}
Furthermore, by \eqref{p1:e3}, with probability at least $1-\delta_\epsilon$, we have
\begin{align*}
    \|AG(C_0)\| & \leq \|G(C_0)\| \leq \xi_\epsilon \|y^\epsilon\| \leq \tfrac{L-1}{4}\epsilon.
\end{align*}
Thus, we arrive at 
\begin{equation*}
    \|(I-AJJ^\top A^\top)^{\tau_{\rm dp}^\epsilon}y\|\leq \tfrac{7L+1}{4}\epsilon \leq 2L\epsilon.
\end{equation*}
Together with the inequality $\sum_{i=1}^n(v,w_i)^2\le \rho^2$,  we finally obtain
\begin{align*}
\sqrt{\sum_{i=1}^n(1-\alpha_i^2\sigma_i^2)^{2\tau_{\rm dp}^\epsilon}\alpha_i^{2\nu}(v,w_i)^2}\le (2L)^\frac{\nu}{\nu+1} \epsilon^\frac{\nu}{\nu+1} \rho^\frac{1}{\nu+1}.
\end{align*}
Now set $\tau:=\tau^\epsilon_{\rm dp}-1$, and we may assume $\tau\ge 0$. Then, the estimate \eqref{t0:e3} implies
\begin{align*}
\tfrac{L+1}{2} \epsilon &<\|\tilde{r}^\epsilon_\tau\| \le \|\tilde{r}_\tau\| + \|\tilde{r}^\epsilon_\tau-\tilde{r}_\tau\|\le \|\tilde{r}_\tau\| + \epsilon \\
&\leq \|(I-AJJ^\top A^\top)^{\tau^\epsilon_{\rm dp}}y\| + \|(I-AJJ^\top A^\top)G(C_0)\| + \epsilon.
\end{align*}
With the argument as in \eqref{eq:approx} and \eqref{eq:approx:2}, it follows that
\begin{align*}
\frac{L-1}{2}\epsilon - \|AG(C_0)\|&\le \sqrt{\sum_{i=1}^n(1-\sigma_i^2\alpha_i^2)^{2\tau^\epsilon_{\rm dp}}\alpha_i^{2(1+\nu)} (v,w_i)^2}\\
&\le \left(\frac{q(1+\nu) B_A^\frac{q+p}{q}}{2e(q+p)b_Ab_\Sigma}
    \right)^\frac{q(1+\nu)}{2(q+p)} \rho (\tau^\epsilon_{\rm dp})^{-\frac{q(1+\nu)}{2(q+p)}}.
\end{align*}
Then by \eqref{p1:e3} and the choice of $\xi_\varepsilon$, with probability at least $1-\delta_\epsilon$,
there holds
\begin{equation*}
\tau^\epsilon_{\rm dp} \le \frac{q(1+\nu) B_A^\frac{q+p}{q}}{2e(q+p)b_Ab_\Sigma}\left(\frac{4\rho}{(L-1)\epsilon}\right)^\frac{2(p+q)}{q(1+\nu)}.
\end{equation*}
Similarly, by Lemma \ref{lem:basic-est} (in the appendix), we have
\begin{align}
\mathcal{E}_1&=\sqrt{\sum_{i=1}^n \frac{(1-(1-\alpha_i^2\sigma_i^2)^{\tau^\epsilon_{\rm dp}})^2}{\alpha_i^2}(w_i',y^\epsilon-y)^2}\le \epsilon \max_{1\le i \le n} \frac{(1-(1-\alpha_i^2\sigma_i^2)^{\tau^\epsilon_{\rm dp}})}{\alpha_i}\nonumber\\
&\le \epsilon b_A^{-\frac{1}{2}}\left(B_AB_\Sigma\right)^\frac{q}{2(q+p)}\sup_{0<\lambda\le 1}\frac{(1-(1-\lambda)^{\tau^\epsilon_{\rm dp}})}{\lambda^{\frac{q}{2(q+p)}}}\le \epsilon b_A^{-\frac{1}{2}}\left(B_AB_\Sigma\right)^\frac{q}{2(q+p)} (\tau_{\rm dp}^{\epsilon})^\frac{q}{2(p+q)}.\label{eqn:E1}
\end{align}
Finally,  combining the previous estimates yields
\begin{align*}
\mathcal{E}_1+\mathcal{E}_2\le&\left(B_AB_\Sigma\right)^\frac{q}{2(q+p)}b_A^{-\frac{1}{2}} \epsilon (\tau^\epsilon_{\rm dp})^\frac{q}{2(p+q)}+\left(2L\right)^\frac{\nu}{\nu+1} \epsilon^\frac{\nu}{\nu+1} \rho^\frac{1}{\nu+1}\\
\le& \left(\left(2L\right)^\frac{\nu}{\nu+1} + \left(\frac{q(1+\nu) B_A^\frac{2q+p}{q}B_\Sigma}{2e(q+p)b_A^\frac{2q+p}{q}b_\Sigma}\right)^\frac{q}{2(q+p)}\left(\frac{4\rho}{L-1}\right)^\frac{1}{1+\nu} \right)  \epsilon^\frac{\nu}{\nu+1} \rho^\frac{1}{\nu+1}.
\end{align*}

\noindent{\bf Step 5: Combining Steps 1 to 4}. By combining the previous estimates, we obtain the desired error bound from Theorem \ref{s2:t2} by
\begin{align*}
&\|G(C^\epsilon_{\tau^\epsilon_{\rm dp}}) - x^\dag\|\le \sum_{i=1}^3 \mathcal{E}_i + 3\xi_\epsilon \|y^\epsilon\|\\ 
\le& 4\xi_\epsilon\|y^\epsilon\| +  \left(\left(2L\right)^\frac{\nu}{\nu+1} + \left(\frac{q(1+\nu) B_A^\frac{2q+p}{q}B_\Sigma}{2e(q+p)b_A^\frac{2q+p}{q}b_\Sigma}\right)^\frac{q}{2(q+p)}\left(\frac{4\rho}{L-1}\right)^\frac{1}{1+\nu} \right) \epsilon^\frac{\nu}{\nu+1}\rho^\frac{1}{\nu+1} \\
\le& \tilde L (\epsilon + \epsilon^\frac{\nu}{\nu+1}\rho^\frac{1}{\nu+1}),
\end{align*}
with probability at least $1-4\delta_\epsilon$. Thus, the proof of the theorem is finished by setting 
\begin{equation*}
    \tilde{L}:= \left(2L\right)^\frac{\nu}{\nu+1} + \left(\frac{q(1+\nu) B_A^\frac{2q+p}{q}B_\Sigma}{2e(q+p)b_A^\frac{2q+p}{q}b_\Sigma}\right)^\frac{q}{2(q+p)}\left(\frac{4\rho}{L-1}\right)^\frac{1}{1+\nu} +L-1,
\end{equation*}
which depends only on the parameters $b_A$, $B_A$, $b_\Sigma$, $B_\Sigma$, $p$, $q$, $\nu$, and $L$.
\end{proof}

\begin{remark}
{\color{black}Under the conditions in Theorem \ref{t0}, one can derive an \textit{a priori} stopping rule. Indeed, from \eqref{eqn:E1}, the data propagation error is bounded by 
\begin{align*}
\sqrt{\sum_{i=1}^n \frac{(1-(1-\alpha_i^2\sigma_i^2)^{\tau})^2}{\alpha_i^2}(w_i',y^\epsilon-y)^2}\le \epsilon b_A^{-\frac{1}{2}}\left(B_AB_\Sigma\right)^\frac{q}{2(q+p)} \tau^\frac{q}{2(p+q)}.
\end{align*}
Similarly, with Lemma \ref{lem:basic-est2}, the linear approximation error is bounded by
\begin{align*}
\sqrt{\sum_{i=1}^n(1-\alpha_i^2\sigma_i^2)^{2\tau^\epsilon_{\rm dp}}\alpha_i^{2\nu}(v,w_i)^2}&\leq \rho\frac{B_A^\frac{\nu}{2}}{(b_Ab_\Sigma)^\frac{q\nu}{2(q+p)}}\sup_{0\le \lambda\leq 1}(1-\lambda)^\tau\lambda^\frac{q\nu}{2(q+p)} \\
& \leq \rho \frac{B_A^\frac{\nu}{2}}{(b_Ab_\Sigma)^\frac{q\nu}{2(q+p)}} \Big(\frac{q\nu }{2e(q+p)}\Big)^\frac{q\nu}{2(q+p)}\tau^{-\frac{q\nu}{2(q+p)}}.
\end{align*}
By choosing $\tau^*=\Big(\frac{\rho^2B_Ab_A}{e^2}\Big)^\frac{q+p}{(1+\nu)q} \Big(\frac{q\nu}{2e(q+p)}\Big)^\frac{\nu}{1+\nu}\frac{1}{(B_AB_\Sigma)^\frac{1}{1+\nu}(b_Ab_\Sigma)^\frac{\nu}{1+\nu}}$ to balance the two components, we obtain that with probability at least $1-3\delta_\epsilon$,
\begin{align*}
    \|G(C^\epsilon_{\tau^*})-x^\dag\|\leq \widehat{L}(\epsilon + \epsilon^\frac{\nu}{1+\nu}\rho^\frac{1}{1+\nu}),
\end{align*}
with the constant $\widehat{L}$ given by 
$$\widehat{L}=L-1+2\Big(\frac{B_AB_\Sigma}{b_Ab_\Sigma}\Big)^{\frac{q\nu}{2(q+p)(1+\nu)}}\Big(\frac{B_A}{b_A}\Big)^\frac{\nu}{2(1+\nu)}\Big(\frac{2\nu(q+p))}{eq}\Big)^\frac{q\nu}{2(q+p)(1+\nu)}.$$
This rate is comparable with that for the discrepancy principle in Theorem \ref{t0}.
}    
\end{remark}

\section{Numerical experiments and discussions} \label{s4}

In this section, we present some numerical results to illustrate the theoretical analysis. The purpose of the experiments is twofold: (i) to validate the convergence rate analysis in a model setting, and (ii) to demonstrate the feasibility of the discrepancy principle \eqref{eqn:dp} in practically more relevant settings. 

First, to validate the convergence rate numerically, we construct the two-layer convolutional generator $G$ and the forward operator $A$ as follows. We set $n=2^6$ and $k=n^2=2^{12}$. We choose the convolutional matrix $U_*$ with $*\in\{{\rm r, s}\}$ (the subscripts r and s stand for rough and smooth, respectively) such that the matrix $\Sigma(U_*)\in\R^{n\times n}$ is diagonal with entries $(\Sigma(U))_{ii} = i^{-p}$, and take a rough architecture with $p_{\rm r}=\frac{1}{2}$ and a smooth one with $p_{\rm s}= \frac{3}{2}$, respectively (Note that eventually we have $\|U\|={\sqrt{2}}$.). For the forward model $A_\cdot\in\mathbb{R}^{n\times n}$ with $\cdot\in\{{\rm a, na}\}$ (the subscripts a and na stand for aligned and non-aligned, respectively), we consider an aligned one {\color{black}$A_{\rm a}$, which is a diagonal matrix with the $i$-th diagonal entry $i^{-\frac{q}{2}}$, and a nonaligned one $A_{\rm na}=HA_{\rm a}H^\top$, where $H\in\R^{n\times n}$ is uniformly sampled from the set of orthogonal matrices}\footnote{\color{black}More precisely, we set $H=UV^\top$, where $N\in\R^{n\times n}$ has i.i.d. standard Gaussian entries with singular value decomposition $N=U\Sigma V^\top$.}. The parameter $q$ is set to $4$. This yields a condition number of ${\rm cond}(A) = {\sigma_1}/{\sigma_n} = 2^{12}$. Then we define the true solution $x^\dag$ for the aligned and non-aligned models to be $x^\dagger_\cdot := A_\cdot^\top A_\cdot w$, with $w\in\R^n$ is the one vector (i.e., $w_i=1$ for all $i=1,...,n$). The exact data $y^\dagger_{ \cdot}$ is $A_{ \cdot} x^\dagger_{ \cdot}$. {\color{black}This procedure ignores the potential discretization errors etc. as is commonly encountered in practice but does allow easily constructing reference solutions satisfying the source-wise representation \eqref{eqn:source}.}
The additive measurement errors $\xi_{\cdot}\in\R^n$ consist of centred i.i.d. Gaussians with variance given by
$\sigma_{\cdot}^2 = \| y^\dagger_{\cdot}\|^2/({n {\rm SNR}^2})$, where ${\rm SNR}\in\{1,3,3^2,3^3,3^4\}$ is the signal-to-noise ratio. The noisy measurements $y^\epsilon_\cdot$ are then given by
$y^\epsilon_\cdot:=y^\dagger_\cdot + \xi_\cdot$,
with the noise level $\epsilon_\cdot:=\|\xi_\cdot\|$. The initial weight matrices $C_0^\cdot \in \R^{n \times k}$ consist of centred i.i.d. Gaussians with variance
$\omega_\cdot^2:={\sigma^2_\cdot}/{\sqrt{n}}$.
For each possible combination of $\{\rm a, na\}$ and $\{\rm s, r\}$, we carry out $\tau_{\rm max}=1500$ iterates to learn the parameter ${C}$, which we denote by $C_\tau^{\cdot,*}$ with $\cdot\in\{{\rm a, na}\}$ and $*\in\{{\rm s, r}\}$. Then we apply the discrepancy principle \eqref{eqn:dp} to choose the stopping index $\tau_{\rm dp}^{\cdot,*}$, and use the index $\tau_{\rm min}^{\cdot,*}$ with the minimal error for comparison:
 \begin{align*}
 \tau_{\rm dp}^{\cdot,*}:&=\min\{\tau \le \tau_{\max}:~\|A_{\cdot}G_{*}(C_{\tau}^{\cdot,*}) - y^\epsilon_{\cdot}\|\le L \epsilon_\cdot\}\quad \mbox{and}\quad
 \tau_{\rm min}^{\cdot,*}:=\arg\min_{\tau \le \tau_{\max}}\|G_*(C_\tau^{\cdot,*}) - x^\dagger_\cdot\|,
 \end{align*}
with the fudge parameter $L=1.05$, {\color{black}as stated in Theorem \ref{t0} (but in practice, one may get better results with the discrepancy principle \eqref{eqn:dp} with some $L<1$; see, e.g. \cite{BerteroBoccacci:2010} for interesting discussions).} For each case, we repeat the experiment $20$ times (different realizations of random noise and initialization) and display the sample mean with the standard deviation (shown in brackets) of the stopping indices and the corresponding relative errors $e_{\min}^{\cdot,*}:=\|G_*(C_{\tau^{\cdot,*}_{\min}})-x^\dagger_\cdot\|/\|x^\dagger_\cdot\|$ and $e_{\rm dp}^{\cdot,*}:=\|G_*(C_{\tau^{\cdot,*}_{\rm dp}})-x^\dagger_\cdot\|/\|x^\dagger_\cdot\|$ in Tables \ref{tab:err-aligned} -- \ref{tab:ind-nonaligned}.
 
{\color{black}We first discuss the aligned model covered by the theory. From Table \ref{tab:err-aligned}, we can clearly observe the convergence of the iterations towards the true solution $x^\dag$ for increasing ${\rm SNR}$ (i.e., decreasing noise level). The error $e_{\rm dp}^{\rm \cdot,*}$ obtained with the discrepancy principle \eqref{eqn:dp} is about $5-100$ percent higher than the minimal error $e_{\rm min}^{\cdot,*}$, with the (relative) margin being larger for smooth $U_{\rm s}$. By comparing the errors obtained for the aligned model $A_{\rm a}$ with smooth $U_{\rm s}$ with that for the aligned model $A_{\rm a}$ with rough $U_{\rm r}$, the latter are more than twice as large. Further, the numerical results are also largely consistent with the convergence rate predicted by Theorem \ref{t0}. Indeed, Theorem \ref{t0} gives the relation ${ e} \sim \epsilon^\frac{\nu}{\nu+1}$ (for each fixed true solution), which yields an affine linear relation $\log({ e}) \sim \frac{\nu}{\nu+1} \log(\epsilon) + {\rm const.}$. Since $\epsilon \sim {\rm SNR}^{-1},$ we perform a linear regression for $\log({\rm SNR}^{-1})$ versus $\log({ e}_{\rm min}^{\rm as})$ and $\log({ e}_{\rm min}^{\rm ar})$, respectively. The resulting estimated smoothness parameters are $\nu_{\rm as} \approx 1.16$ and $\nu_{\rm ar} \approx 1.04$ with corresponding adjusted $\mathcal{R}^2$ values of $0.99891$ and $0.99995$, respectively (recall that a perfect linear relation would yield $1.0$). Meanwhile, from  Table \ref{tab:ind-aligned} it can be observed that the indices $\tau_{\rm min}^{\rm a,*}$ and $\tau_{\rm dp}^{\rm a,*}$ are significantly smaller when using the rough $U_{\rm r}$ instead of the smooth $U_{\rm s}$: That is, in the case of rough $U_{\rm r}$, gradient descent requires fewer iterations to reach both the stopping criterion and the optimal index. In contrast, the results for the nonaligned model exhibit different observations, as is indicated by Table \ref{tab:err-nonaligned}. Here, for smooth $U_{\rm s}$, the minimal error stays effectively constant for all noise levels. In contrast, for rough $U_{\rm r}$, we observe convergence, however, applying linear regression after taking the logarithms (only for ${\rm SNR}\in\{1,...,27\}$, since clearly for ${\rm SNR}=81$, the maximal number of iterations is too small) yields a much smaller estimated smoothness parameter $\nu_{\rm nr}\approx 0.13$, where the corresponding adjusted $\mathcal{R}^2$ value of $0.94691$ indicates that the relation ${ e} \sim \epsilon^\frac{\nu}{\nu+1}$ probably is not valid here. While the behaviour of the stopping indices for rough $U_{\rm r}$ in Table \ref{tab:ind-nonaligned} is similar to the aligned setting, it is completely different for smooth $U_{\rm s}$. Additional experiments (not reported)  indicate that a much higher number of iterations is required in this case.} 
Similarly, the stopping indices for both rough $U_{\rm r}$ and smooth $U_{\rm s}$ are substantially smaller for the aligned forward model $A_{\rm a}$ than for the non-aligned one $A_{\rm na}$. Since the smoothness of the true solution $x^\dag_\cdot$ relative to the respective forward model $A_\cdot$ is the same in all cases, this behavior is attributed to the fact that by construction the condition number of $A_\cdot J_{\rm s}$ is larger than that of $A_\cdot J_{\rm r}$ (recall that $\Sigma(U_*) = J_*J_*^\top$), and likewise the condition number of $A_{\rm a}J_*$ is larger than that of $A_{\rm na}J_*$. Also, the stopping index $\tau_{\rm dp}^{\cdot,*}$ of the discrepancy principle \eqref{eqn:dp} is usually smaller than the minimizing index $\tau_{\rm min}^{\cdot,*}$. This means that the accuracy loss of the discrepancy principle \eqref{eqn:dp} is actually due to stopping too early (and thus in practice one may choose $L$ very close to $1$), and not to a lack of stability. Note that in the experiments we have chosen the fudge parameter $L$ relatively close to one, which is justified since we are using the discrepancy principle \eqref{eqn:dp} with the exact error bound. If this bound has to be estimated, as is often the case in practice, one typically has to choose a larger $L$. 

\begin{table}[hbt!]
\centering
\caption{Relative error for aligned $A_{\rm a}$.\label{tab:err-aligned}}
\setlength{\tabcolsep}{4pt}
\begin{tabular}{c|cc|cc|}
\toprule
\multicolumn{1}{c}{}&\multicolumn{2}{c}{smooth $U$} & \multicolumn{2}{c}{rough $U$} \\
\cmidrule(l){2-3} \cmidrule(l){4-5} 
${\rm SNR}$ & $e_{\min}$  & $e_{\rm dp}$ & $e_{\min}$ & $e_{\rm dp}$ \\
\midrule
1& 1.52e-1  (5.9e-2) & 1.67e-1  (6.0e-2) & 3.89e-1  (4.8e-2) & 3.96e-1  (4.8e-2) \\
 3&7.98e-2  (2.6e-2) & 1.10e-1  (3.2e-2) & 2.19e-1  (2.5e-2) & 2.33e-1  (2.9e-2) \\
 $3^2$&4.66e-2  (1.4e-2) & 7.01e-2  (1.6e-2) & 1.26e-1  (1.4e-2) & 1.36e-1  (1.7e-2) \\
 $3^3$ & 2.48e-2  (6.0e-3) & 5.13e-2  (9.9e-2) & 7.19e-2  (7.2e-3) & 8.34e-2  (9.2e-3) \\
 $3^4$ & 1.43e-2  (3.9e-3) & 2.29e-2  (4.2e-3) & 4.12e-2  (4.3e-3) & 4.45e-2  (4.7e-3) \\
\bottomrule
\end{tabular}
\end{table}

\begin{table}[hbt!]
\centering
\caption{Stopping index for aligned $A_{\rm a}$.\label{tab:ind-aligned}}
\setlength{\tabcolsep}{4pt}
\begin{tabular}{c|cc|cc|}
\toprule
\multicolumn{1}{c}{}&\multicolumn{2}{c}{smooth $U$} & \multicolumn{2}{c}{rough $U$} \\
\cmidrule(l){2-3} \cmidrule(l){4-5} 
${\rm SNR}$ & {\color{black}$\tau_{\min}$}  & {\color{black}$\tau_{\rm dp}$} & {\color{black}$\tau_{\min}$} & {\color{black}$\tau_{\rm dp}$} \\
\midrule
  1 &  9.1   (21.1) &  1.0  (0.0) &   4.7    (8.1) &  1.0 (0.0) \\
  3 & 16.4   (33.5) &  1.0  (0.0) &  11.8   (19.8) &  1.0  (0.0) \\
  $3^2$& 63.5   (62.2) &  2.0  (0.2) &  26.4   (35.2) &  2.0  (0.2) \\
 $3^3$ & 137.6  (106.9) & 10.1  (6.0) &  61.1   (62.0) &  6.1  (2.7) \\
 $3^4$ & 302.5  (289.0) & 58.9  (6.8) & 111.8  (114.1) & 28.6  (4.4) \\
\bottomrule
\end{tabular}
\end{table}

\begin{table}[hbt!]
\centering
\caption{Relative error for non-aligned $A_{\rm na}$.\label{tab:err-nonaligned}}
\setlength{\tabcolsep}{4pt}
\begin{tabular}{c|cc|cc|}
\toprule
\multicolumn{1}{c}{}&\multicolumn{2}{c}{smooth $U$} & \multicolumn{2}{c}{rough $U$} \\
\cmidrule(l){2-3} \cmidrule(l){4-5} 
${\rm SNR}$ & $e_{\min}$  & $e_{\rm dp}$ & $e_{\min}$ & $e_{\rm dp}$ \\
\midrule
 $1$ & 9.91e-1  (1.1e-2) & 2.73e0  (3.2e-1) & 7.10e-1  (6.9e-2) & 7.35e-1  (6.1e-2) \\
 $3$ & 9.85e-1  (6.5e-3) & 3.19e0  (1.2e-1) & 5.96e-1  (5.4e-2) & 6.64e-1  (4.9e-2) \\
 $3^2$ & 9.83e-1  (3.9e-3) & 2.98e0  (5.0e-2) & 5.18e-1  (2.6e-2) & 5.48e-1 (3.6e-2) \\
 $3^3$ & 9.82e-1 (2.4e-3) & 2.89e0  (3.4e-2) & 4.85e-1  (1.9e-2) & 4.99e-1  (1.3e-2) \\
 $3^4$ & 9.81e-1  (1.5e-3) & 2.86e0  (2.8e-2) & 4.75e-1  (9.4e-3) & 4.81e-1  (1.1e-2) \\
\bottomrule
\end{tabular}
\end{table}

\begin{table}[hbt!]
\centering
\caption{Stopping index for non-aligned $A_{\rm na}$.\label{tab:ind-nonaligned}}
\setlength{\tabcolsep}{4pt}
\begin{tabular}{c|cc|cc|}
\toprule
\multicolumn{1}{c}{}&\multicolumn{2}{c}{smooth $U$} & \multicolumn{2}{c}{rough $U$} \\
\cmidrule(l){2-3} \cmidrule(l){4-5} 
${\rm SNR}$ & {\color{black}$\tau_{\min}$}  & {\color{black}$\tau_{\rm dp}$} & {\color{black}$\tau_{\min}$} & {\color{black}$\tau_{\rm dp}$} \\
\midrule
1 & 1.4  (0.8) &   48.4   (16.5) &   20.8   (25.9) &   5.2    (0.8) \\
3 & 1.5  (0.6) &  433.9  (238.9) &  127.7   (81.5) &  15.6    (7.6) \\
$3^2$ & 1.6  (0.5) & 1467.2  (105.8) &  439.8  (357.3) &  82.7   (18.1) \\
 $3^3$ & 1.9  (0.2) & 1500.0    (0.0) & 1100.2  (569.6) & 192.7   (26.6) \\
 $3^4$ & 2.0  (0.0) & 1500.0   (0.0) & 1500.0    (0.0) & 848.6  (441.8) \\
\bottomrule
\end{tabular}
\end{table}

In summary, our small experimental study verifies the theoretical results of Theorems \ref{t0} and \ref{s2:t2}. Regarding the alignment assumption of the singular vectors used in Theorem \ref{t0} to derive a convergence rate, the experiments indicate that the assumption is indeed critical for the success of the method. For the choice of the convolutional matrix $U$, a smoother $U$ may give additional accuracy, but at the cost of increasing the number of iterations (and thus the numerical cost) required to achieve that accuracy. 

Finally, we illustrate the discrepancy principle \eqref{eqn:dp} on a more realistic setting. Obviously, the two-layer convolutional generator analyzed in this paper is not well suited for practical applications due to its simplicity. Therefore, we investigate the application of the discrepancy principle \eqref{eqn:dp} to a practically competitive architecture, i.e., the deep decoder \cite{Heckel:2018deep}. This network was also used for the numerical examples in \cite{HeckelSoltanolkotabi:2020denoising,HeckelSoltanolkotabi:2020CS}, where the two-layer generator was theoretically analyzed in related settings. We choose the same network structure as in \cite{HeckelSoltanolkotabi:2020CS}, with four layers. To minimize the loss, we use  {\color{black}  gradient descent, with two empirically chosen constant learning rates (LRs) of 1e-4 and 4e-4.}
Specifically, we consider the task of deblurring a corrupted image of size $128 \times 128$, blurred by a Gaussian kernel (realized by {\tt v2.GaussianBlur} from \texttt{torchvision} with kernel size $(11,11)$ and variance $\sigma=1.5$). In Figure \ref{fig:image} we plot the true image, the noisy blurred image, and the reconstructed images using the discrepancy principle \eqref{eqn:dp} for both LRs. In Figure \ref{fig:traing} we show the evolution of the residual $\|AG(C^\epsilon_\tau)-y^\epsilon\|$ and the error $\|G(C^\epsilon_\tau)-x^\dagger\|$ during the entire training process. The stopping index $\tau_{\rm dp}^*$ of the discrepancy principle \eqref{eqn:dp} is given by the point where the red bar intersects the iteration trajectory for the first time. These plots show that the discrepancy principle \eqref{eqn:dp} can indeed give a reliable reconstruction. {\color{black} The trajectories indicate that a smaller LR yields a smoother trajectory, while a larger LR leads to faster convergence. Moreover, the solutions obtained with the two LRs are different. Nonetheless, due to the ill-posedness of the deblurring task, further iterations beyond $\tau_{\rm dp}^*$ can potentially lead to overfitting and also to larger errors in the reconstructions}. Thus, when applying overparameterized neural networks to ill-posed problems, early stopping is indeed critical, as previously observed \cite{Ulyanov:2018,Barbano:2023subspace}. {\color{black}Although not presented, this phenomenon is also observed for the popular ADAM algorithm \cite{KingmaBa:2015}.}

\begin{figure}[hbt!]
\centering\setlength{\tabcolsep}{0pt}
\begin{tabular}{cc}
    \includegraphics[width=0.45\linewidth] {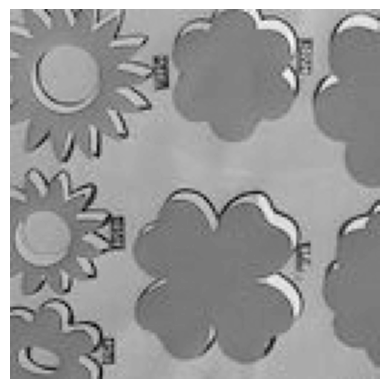}
&    \includegraphics[width=0.45\linewidth]{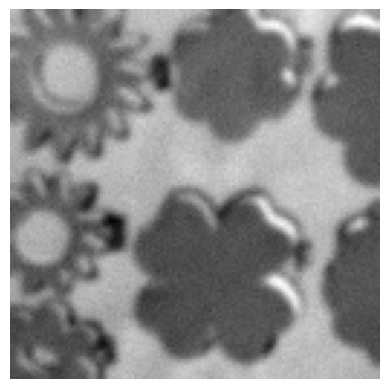}\\
(a) true image $x^\dag$ & (b) noisy blurred image $y^\epsilon$ \\
    \includegraphics[width=0.45\linewidth]{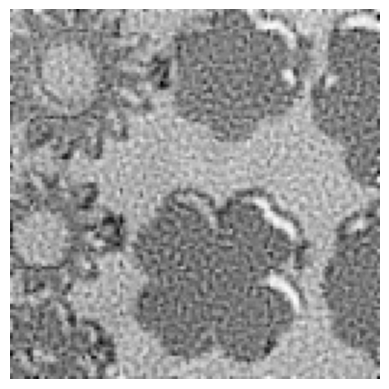}

&    \includegraphics[width=0.45\linewidth]{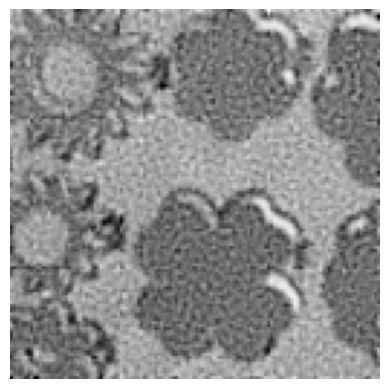}\\
{\color{black}
(c) reconstruction $G(C^\epsilon_{\tau_{\rm dp}^\epsilon})$ (LR=1e-4)} &{\color{black} (d) reconstruction $G(C^\epsilon_{\tau_{\rm dp}^\epsilon})$  (LR=4e-4)
}\end{tabular}
\caption{Ground truth, noisy data and reconstructions  for image deblurring.}\label{fig:image}
\end{figure}

\begin{figure}[hbt!]
\centering\setlength{\tabcolsep}{0pt}
\begin{tabular}{cc}
    \includegraphics[width=.48\linewidth]{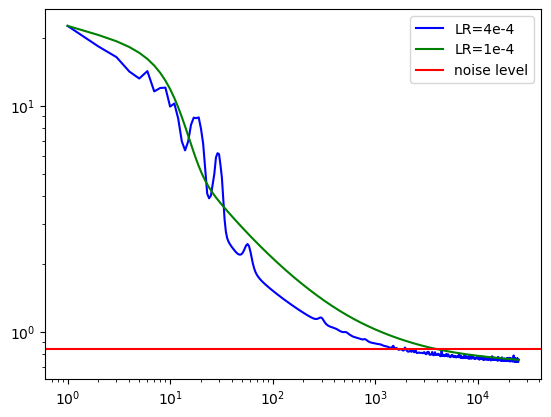}
  &  \includegraphics[width=.5\linewidth]{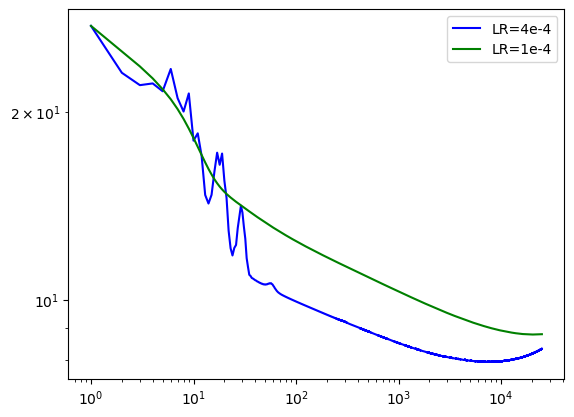}\\
(a) residuals $\|AG(C_\tau^\epsilon) - y^\epsilon\|$ & (b) reconstruction errors $\| G(C_\tau^\epsilon) - x^\dagger\|$
\end{tabular}
\caption{Training dynamics of the DIP reconstruction: (a) residual errors  $\|AG(C_\tau^\epsilon) - y^\epsilon\|$ (and noise level $\epsilon$, indicated by the red horizontal line) and (b) reconstruction errors $\| G(C_\tau^\epsilon) - x^\dagger\|$.}\label{fig:traing}
\end{figure}

\section{Concluding Remarks}

In this work, we have investigated the use of a two-layer convolutional neural network as a regularization method for solving linear ill-posed problems, in which the weights of the neural network are tuned without using training data by minimizing a least-squares loss for a single measurement. We have analyzed the data-driven discrepancy principle as a stopping rule, and derive optimal convergence rates (with high probability) under the canonical source condition. We presented several numerical experiments to confirm the theoretical results for the two-layer network, and also provided additional simulations on deblurring with competitive network architectures to illustrate the feasibility of the approach.

There are several open research questions. For example, more advanced techniques such as stochastic gradient descent, ADAM, or quasi-Newton methods could be employed instead of using a simple gradient descent algorithm for minimization. In addition, it would be of interest to extend the analysis to multi-layer architectures. Finally, the main challenge is to improve the analysis for more realistic settings. Currently, it is based on the fact that the nonlinear minimization problem can be approximated by a nearby linear problem if the neural network is wide enough. However, the required width can be huge. For example, the given lower bound cannot be satisfied in any numerical application. Nevertheless, the numerical experiments show that the network still produces a good reconstruction. Therefore, there must be an alternative analysis with less stringent requirement on the width. 

\section*{Acknowledgements} The authors are grateful to the referees for their constructive comments which heave led to an improvement in the quality of the paper.

\appendix
\section{Technical lemma}
\begin{lemma}\label{lem:basic-est}
For $s>0$, and $\tau \in\mathbb{N}$, the following inequality holds
\begin{equation*}
    \sup_{0<\lambda \leq 1}\lambda^{-s}(1-(1-\lambda)^\tau) \leq \tau^s
\end{equation*}
\end{lemma}
\begin{proof}
Clearly, the function satisfies the desired bound for $\lambda >\frac{1}{\tau}$ so it suffices to bound it on $(0,\frac{1}{\tau}]$. To this end, let $f(\lambda)=(1-\lambda)^\tau$. Since $f$ is twice differentiable over $(0,\frac{1}{\tau})$, the first-order Taylor expansion of $f$ at $\lambda=0$ with the exact Pearno remainder is given by
\begin{align*}
    f(\lambda) & = f(0) + f'(0)\lambda + \frac{f'(\lambda_0)}{2}\lambda^2 
     = 1 - \tau \lambda + \frac{\tau(\tau-1)}{2}(1-\lambda_0)^{\tau-2}\lambda^2,
\end{align*}
with $\lambda_0\in (0,\frac1\tau)$. Thus, 
\begin{align*}
    & \lambda^{-p}(1-(1-\lambda)^\tau)  = \lambda^{-p}\left(1-\left(1 - \tau \lambda + \frac{\tau(\tau-1)}{2}(1-\lambda_0)^{\tau-2}\lambda^2\right) \right)\\
      =& \tau \lambda^{1-p} - \frac{\tau(\tau-1)}{2}(1-\lambda_0)^{\tau-2}\lambda^{2-p} \leq \tau\cdot (\tau^{-1})^{1-p} = \tau^p,
\end{align*}
since $\lambda \leq \tau^{-1}$, and the second term is non-positive.
\end{proof}
\begin{lemma}\label{lem:basic-est2}
For any $r> 0$, and $\tau>0$ with $\tau >r$, the following estimate holds
\begin{equation*}
    \sup_{0\leq \lambda \leq 1} (1-\lambda)^\tau \lambda^r \leq \left(\frac{r}{e}
    \right)^r \tau^{-r}.
\end{equation*}
\end{lemma}
\begin{proof}
It follows directly from the inequality $1-x\leq e^{-x} $ for any $x\in\mathbb{R}$ that
$\sup_{0\leq \lambda \leq 1} (1-\lambda)^\tau \lambda^r \leq e^{-\tau \lambda} \lambda^r.$
Let $f(\lambda)=e^{-\tau \lambda} \lambda^r$. Then $
    f'(\lambda)=-\tau e^{-\tau\lambda } \lambda^r+ e^{-\tau\lambda }r\lambda^{r-1} = (-\tau \lambda + r) e^{-\tau\lambda } \lambda^{r-1}.$ 
Obviously, $\lambda^*=\frac{r}{\tau}$ is the unique zero of $f'$ on $(0,1]$, and since $f(\lambda)$ is nonnegative and $f(0)=0$. it follows that the maximum of $f$ is achieved at $\lambda^*$, and the maximum is $ e^{-r}\left(\frac{r}{\tau} \right)^r$. This completes the proof.
\end{proof}

\bibliographystyle{siamplain}
\bibliography{unn.bib}

\end{document}